\documentclass[12pt]{article}
{\begin{Sbox}\begin{minipage}}%
{\end{minipage}\end{Sbox}\fbox{\TheSbox}}%

\usepackage[psamsfonts]{amssymb}
\usepackage{amsmath, amsthm, anysize, enumerate, color, url}
\usepackage{graphicx}
\usepackage{epstopdf}
\usepackage{epsfig}
\usepackage{subfigure}
\usepackage[ruled,linesnumbered]{algorithm2e}
\usepackage[sectionbib]{natbib}

\usepackage{threeparttable}

\usepackage[colorlinks, breaklinks,
                   linkcolor=red,
                  citecolor=blue
                  ]{hyperref}

\usepackage{breakurl}

\newtheorem{theorem}{Theorem} %%%[section]
\newtheorem{corollary}{Corollary}

\newtheorem{lemma}{Lemma} %%%[theorem]
\newtheorem{proposition}{Proposition} %%%[theorem]

\theoremstyle{definition}
\newtheorem{assumption}{Assumption}

\newtheorem{example}{Example}
\newcommand{\E}{\mathbb{E}}
\newcommand{\ep}{\mathbb{E}}
\newcommand{\R}{\mathbb{R}}

\renewcommand{\P}{\mathbb{P}}
\newcommand{\pr}{\mathbb{P}}

\newcommand{\diag}{\mbox{diag}}
\newcommand{\bs}{\boldsymbol}

\renewcommand{\hat}{\widehat}
\renewcommand{\tilde}{\widetilde}

\definecolor{yuancolor2}{rgb}{0, 0, 0}
\definecolor{yuancolor}{rgb}{0, 0, 0}
\newcommand{\yuan}[1]{{\color{yuancolor}#1}}

\usepackage{setspace}

\begin{document}
\title{Asymptotic theory in network models with covariates and a growing number of node parameters}

\author{Qiuping Wang\thanks{Both authors  contributed equally to this work.} \thanks{School of Mathematics and Statistics,
        Zhaoqing University,
        Zhaoqing  526061, China.
\texttt{Email:} qp.wang@mails.ccnu.edu.cn.}
\hspace{5mm}
Yuan Zhang$^*$\thanks{Department of Statistics, The Ohio State University, Columbus, 43210, U.S.A.
\texttt{Email:} yzhanghf@stat.osu.edu}
\hspace{5mm}
Ting Yan\thanks{Department of Statistics, Central China Normal University, Wuhan, 430079, China.
\texttt{Email:} tingyanty@mail.ccnu.edu.cn.}
\\
$^*$Zhaoqing University,  \\
$^\dag$The Ohio State University, \\
$^\ddag$Central China Normal University
}
\date{}

\maketitle

\begin{abstract}
We \yuan{propose} a general model \yuan{that jointly characterizes} degree heterogeneity and homophily in weighted, undirected networks.
We present \yuan{a} moment estimation \yuan{method using} node degrees and \yuan{homophily statistics}.
\yuan{We establish consistency and asymptotic normality of our estimator using novel analysis.
We apply our general framework to three applications, including both exponential family and non-exponential family models.  Comprehensive numerical studies and a data example also demonstrate the usefulness of our method.
}
% {When all the parameters are bounded}, we have $\|\widehat\beta-\beta\|_\infty=O_p\left( n^{-1/2} \right)$ and $\|\widehat{\gamma} - \gamma\|_\infty=O_p\left( n^{-1} \right)$, up to a logarithm factor.
% Further, we derive a asymptotic representation of the moment estimator, based on which, we derive their asymptotic normal distributions
% under classical CLT conditions.
% Three applications are provided to illustrate the unified theoretical result.
% Numerical studies and a real data analysis demonstrate our theoretical findings.


\end{abstract}

\vskip 5pt

%%Running title: Undirected network models

\section{Introduction}

\yuan{Jointly modeling a network and nodal or edge-wise coviarates has long been an interesting problem.
One natural idea is to extend a widely-used network model to incorporate covariates.
For example, \citep{yang2013community, zhang2016community, binkiewicz2017covariate} introduce nodal covariates into a stochastic block model (SBM), which captures the clustering structure in networks.  In this paper, we will study the extension of another network model, called $\beta$-model \citep{Rinaldo2013, Yan:Xu:2013, Yan:Qin:Wang:2015, Yan-Jiang-Fienberg-Leng2018, chen2021analysis, zhang2021L2} that characterizes a different important aspect of network data, namely, \emph{degree heterogeneity} \citep{Cho:2011}.  The degree (total number of connections) of a node provides important profiling information about its structural role in the network \citep{borgatti2000models,zhang2020edgeworth, maugis2020central}.  A famous example is that \citet{babai1980random} shows that efficient graph matching can usually succeed with high probability between two shuffled random graphs, using a degree-based algorithm.  The $\beta$-model, named by \citet{Chatterjee:Diaconis:Sly:2011}, is an undirected, binary network:
\begin{equation}
    \pr(a_{ij}=1) = \dfrac{e^{\beta_i+\beta_j}}{1+e^{\beta_i+\beta_j}},
    \quad
    a_{ij}=a_{ji},
    \quad
    1\leq i<j\leq n
    \label{model::original-beta-model}
\end{equation}
where $A=(a_{ij}), 1\leq \{i,j\}\leq n$ is a binary adjacency matrix.  Later, \citet{Yan:Qin:Wang:2015, fan2022asymptotic} extend this model to weighted networks with edge distributions including Poisson, geometric, exponential and so on.  This paper will generalize \citet{Yan:Qin:Wang:2015} rather than the original binary edge $\beta$-model.

What we incorporate into a weighted $\beta$-model are \emph{edge-wise covariates}.  We notice that this set up also accommodates nodal covariates (e.g. immutable characteristics such as gender, race and genetic features; and/or mutable ones, including location, occupation and hobbies) since they can be easily transformed into edge-wise similarity/dissimilarity measures.  According to \citet{Graham2017}, this part of the data encodes the \emph{homophily} effect in network formation.  As a quick illustration, consider two node pairs $(i_1,i_2)$ and $(j_1,j_2)$.  Even if $\{\beta_{i_1},\beta_{i_2}\}$ is very different from $\{\beta_{j_1},\beta_{j_2}\}$, their edge expectations $\ep[a_{i_1,i_2}]$ and $\ep[a_{j_1,j_2}]$ might not differ too much, if they have similar edge covariates $z_{i_1,i_2}\approx z_{j_1,j_2}$.

Jointly modeling both degree heterogeneity and homophily, as well as developing effective estimation and inference methods along with supporting theory, is an interesting challenge.  As aforementioned, covariate-assisted stochastic block models have been comparatively well-studied, whereas few works exist to extend the $\beta$-model.  Among notable exceptions, \citet{Graham2017} generalizes \eqref{model::original-beta-model} by extending $\beta_i+\beta_j$ to $\beta_i+\beta_j+z_{ij}^T\gamma$ and devises a likelihood-based method for estimation and inference.
Recently, independent works \citet{stein2020sparse,stein2021sparse} further introduce $\ell_1$ regularization to the joint model for undirected and directed networks, respectively.
Both research groups \citet{Graham2017} and \citet{stein2020sparse,stein2021sparse} focus exclusively on binary edges for cleanness.
On the other hand, many networks, such as communications, co-authorship, brain activities and others, have weighted edges.

In this paper, we develop a general joint model for weighted edges.  Different from the likelihood-based approaches in \citet{Graham2017} and \citet{stein2020sparse,stein2021sparse}, we propose and analyze a method-of-moment parameter estimation.  As discussed later in the paper, moment method has its unique advantage in addressing slightly dependent network edge formation -- despite this paper exclusively focuses on independent edge generation, we understand that a comprehensive study here paves the road towards successfully handling the very challenging problem of dependent edges.
}

\yuan{
 We develop a two-stage Newton method that first finds an error bound for $\| \widehat{\beta} - \beta\|_\infty$ for a fixed $\gamma$ via establishing the convergence rate of  the Newton iterative sequence
 and then derives $\| \widehat{\gamma} - \gamma\|_\infty$ based on a profiled equation under some conditions.
 When all parameters are bounded, the $\ell_\infty$ norm error for $\widehat{\beta}$ is in the order of $O_p(n^{-1/2})$ while
 the $\ell_\infty$ norm error for $\widehat{\gamma}$ is in the order of $O_p(n^{-1})$, both up to a logarithm factor.
 When the parameters diverge, the error bounds depend on additional factors involved with the ranges of $\beta$ and $\gamma$.
Further, we derive an asymptotic representation of the moment estimator, based on which, we derive their asymptotic normal distributions under classical CLT conditions. To illustrate the unified results, we
present three applications, along with comprehensive numerical simulations and a real data example.
}

The rest of the paper is organized as follows.
\yuan{
In Section \ref{section:model}, we present our general model.
In Section \ref{section:estimation}, we propose our moment estimation equations.
In Section \ref{section:asymptotic}, we establish consistency and asymptotic normality of our estimator under mild conditions.
Section \ref{section:app} illustrates the application of our general framework to weighted networks with logistic, Poisson and probit edge formation schemes.
Section \ref{section:sd} contains summary and discussion. Due to limited space, simulation results and the real data application are relegated to Supplementary Material.
}

% We give the summary and further discussion in Section .
% Simulations and a real data analysis are relegated into the Supplementary Material.
% The proof of Theorem \ref{Theorem:con} %% the main results
% are relegated into Section \ref{section:appendix}.
% The proofs of  Theorems \ref{Theorem-central-a} and \ref{theorem-central-b} and all supported lemmas
% are given in the Supplementary Material.

\section{\yuan{Covariate-assisted $\beta$-model}}
\label{section:model}

\yuan{We shall jointly analyze data from two sources: network and edge-wise covariates.  The network data is represented by an adjacency matrix} $A=(a_{ij})_{n\times n}, \yuan{1\leq i<j\leq n}$.  \yuan{We study undirected networks without self-loops, i.e. $A$ is symmetric $a_{ij}=a_{ji}$ and $a_{ii}=0$.  In this paper, each entry $a_{ij}$ may be binary or weighted (such as collaboration counts in a co-authorship network and phone call lengths).}
Let $d_i= \sum_{j \neq i} a_{ij}$ be the degree of node $i$
and $d=(d_1, \ldots, d_n)^T$ be the degree sequence.
% of the graph $G_n$.
\yuan{In addition to network data}, we \yuan{also} observe a covariate vector \yuan{$z_{ij}\in\mathbb{R}^p$ on each edge}.
\yuan{This setting also covers the scenario when we observe nodal attributes $x_i$: simply define a similarity/dissimilarity measure $g(\cdot,\cdot)$ that converts these attributes to edge-wise covariates via $g(x_i,x_j)$.  Examples including Euclidean distance for continuous $x_i$'s and Hamming distance for binary $x_i$'s.}

\yuan{Our goal is to jointly model \emph{degree heterogeneity} and \emph{homophily}.}
\yuan{Degree heterogeneity} is \yuan{captured by a latent} parameter $\beta_i\in\mathbb{R}$ \yuan{on each node}.
\yuan{Homophily is driven by edge-wise covariates under our framework.  Specifically, it is accounted for by $z_{ij}^T \gamma$, where the exogenous parameter $\gamma\in\mathbb{R}^p$ can be understood analogously like a regression coefficient in a generalized linear model.}

\yuan{Now we present our model.  Given $Z=(z_{ij})$, the network entries $a_{ij}$ are generated independently by the following model, which we call ``covariate-assisted $\beta$-model'':}
\begin{equation}
    \label{model-2}
    a_{ij}|\{z, \beta, \gamma\}
    \sim
    f\big(
        a_{ij} \big| \beta_i + \beta_j + z_{ij}^T \gamma
    \big).
\end{equation}
where $f$ is a known probability density\yuan{/mass} function, $\beta_i$ is the degree parameter of node $i$ and $\gamma$ is a $p$-dimensional regression coefficient for the covariate $z_{ij}$.
\yuan{Our model \eqref{model-2} generalizes the semi-parametric models in econometrics literature \citep{FVW2016} with binary and exponential responses for undirected networks.  We focus on these additive models for computational tractability.  It would be an interesting future work to generalize the method of our analysis to address the more general case where $\beta, z$ and $\gamma$ enter the model as non-additive effects.}

\yuan{The model \eqref{model-2} extends not only the well-known $\beta$-model, but also many of its variants \citep{Yan:Qin:Wang:2015}.  In many examples, such as $\beta$-model, $f(\cdot)$ is an increasing function of $\beta_i$.  Consequently, nodes having relatively large degree parameters will have more links than those nodes with low degree parameters, without considering homophily.  To further illustrate the usefulness of \eqref{model-2}, we consider two running examples.
}

\begin{example}(Binary edges)
\label{example-a}
\yuan{In statistical network analysis, a long-studied problem is to jointly model network data with additional covariates.  For example,  \citet{yang2013community, zhang2016community} incorporate nodal covariates into a stochastic block model; in contrast our model provides a flexible tool for incorporating nodal and/or edge-wise covariates into a $\beta$-model.}
\yuan{Here, we consider binary edges, i.e.} $a_{ij}\in \{0,1\}$.  Let
$F$ be \yuan{some properly chosen transformation: $F: \mathbb{R}\to[0,1]$.}
% a cumulative distribution function.
The probability of $a_{ij}$ is
\[
\P( a_{ij}= a ) = \{F( \beta_i+\beta_j + z_{ij}^T \gamma )\}^{a} \{1-F(\beta_i+\beta_j+z_{ij}^T \gamma ) \}^{1-a},~~a\in\{0, 1\}.
\]
Two \yuan{popular choices of} $F(\cdot)$ are \yuan{sigmoid transformation} $F(x)=e^x/(1+e^x)$ \citep{Graham2017} and probit transformation $F(x)=\Phi(x)$, where $\Phi(x)$ is the \yuan{CDF of $N(0,1)$}.
\end{example}

\begin{example}\yuan{(Unbounded discrete edges)}
% (Infinite discrete weight)
\label{example-b}
\yuan{This example generalizes the Poisson $\beta$-model in Section 3.4 of \citet{Yan:Qin:Wang:2015}.  Here, we model $a_{ij}\sim$~Poisson$\big(\lambda_{ij}\big)$, where $\lambda_{ij}= \exp( \beta_i+\beta_j + z_{ij}^T \gamma)$.  That is,}
\[
\log \P(a_{ij}=a) = a(\beta_i+\beta_j + z_{ij}^T \gamma) - \exp ( \beta_i+\beta_j + z_{ij}^T \gamma) - \log a!.
\]
\end{example}

\section{\yuan{Parameter} estimation}
\label{section:estimation}

\yuan{To motivate the estimation method for our covariate-assisted $\beta$-model \eqref{model-2}, it is helpful to very briefly recall that for the classical $\beta$-model, i.e. $\gamma\equiv0$, there are two mainstreams of estimation methods: MLE \citep{Chatterjee:Diaconis:Sly:2011, Rinaldo2013} and method of moments \citep{Yan:Qin:Wang:2015}.
While MLE is a widely-recognized method for model parameter estimation, as aforementioned in the introduction, method-of-moments may enjoy an easier extension to dependent-edge scenarios in the future (see also Section \ref{section:sd}).  Therefore, we propose and study a moment estimator in this paper.
}
\yuan{
Notice that $\ep[a_{ij}]$ depends on the model parameters only through
\begin{equation}\label{definition-pi}
\pi_{ij}:= \beta_i + \beta_j  + z_{ij}^T \gamma.
\end{equation}
There exists a function $\mu(\cdot)$ such that $\ep[a_{ij}] = \mu(\pi_{ij})$.
In some future texts, we find it more convenient to emphasize that $\mu(\pi_{ij})$ can be viewed as the $(i,j)$ element of an $n\times n$ matrix.  Therefore, we slightly abuse notation and might sometimes use a different notation $\mu_{ij}(\beta, \gamma)$ to represent $\mu(\pi_{ij})$.

Now we are ready to present our method of moments parameter estimation.  To motivate our formulation, consider a special case when the edge distribution belongs to an exponential family, namely,
\begin{equation}
    \label{eqn::model-exp-family}
    L(a|\beta,z,\gamma)
    =
    C(\beta,z,\gamma)
    \cdot
    e^{\big(\sum_{1\leq i<j\leq n}a_{ij}
    \cdot (\beta_i+\beta_j+z_{ij}^T\gamma)\big)}
    \cdot h(a,z).
\end{equation}
Examples of \eqref{eqn::model-exp-family} include logistic model and Poisson model with covariates, as in Section \ref{section:app}.
% \commentyuan{[Yuan says: please add a few proper special cases (examples) here, just names.]}
Sending the partial derivatives of the log-likelihood function of \eqref{eqn::model-exp-family} to zero, we obtain the following moment equations.
\begin{align}
    d_i &= \sum_{j:j\neq i} \mu_{ij}(\beta,\gamma),
    \quad i\in [n]
    \label{eq:moment-beta}
    \\
    \sum_{1\leq i<j\leq n} z_{ij}a_{ij}
    &= \sum_{1\leq i<j\leq n} z_{ij}\mu_{ij}(\beta,\gamma)
    \label{eq:moment-gamma}
\end{align}
where $[n]=\{1,\ldots,n\}$.
Denote the solution to \eqref{eq:moment-beta} and \eqref{eq:moment-gamma} by $(\hat\beta,\hat\gamma)$.
We shall address the natural questions of existence and uniqueness of $(\hat\beta,\hat\gamma)$ in Theorem \ref{Theorem:con}.
}

Now we discuss some computational issues.
When the number of nodes $n$ is small and $f$ is the binomial, Probit, or Poisson probability function or Gamma density function,
we can simply use the package ``glm" in the R language to solve \eqref{eq:moment-beta} and \eqref{eq:moment-gamma}.
For relatively large $n$, it might not have large enough memory to store the design matrix for $\beta$ required by the R package ``glm".
In this case,
we recommend the use of a two-step iterative algorithm by alternating between solving the first equation in  \eqref{eq:moment-beta} via the fixed point method [\citet{Chatterjee:Diaconis:Sly:2011}]
or {the gradient descent algorithm [\citet{MAL-050}]}
and solving the second equation in \eqref{eq:moment-gamma}.

\section{Asymptotic properties}
\label{section:asymptotic}
In this section, we present the consistency and asymptotic \yuan{normality} of the moment estimator.
\yuan{We start with notation.}
For \yuan{any} $C\subset \R^n$, let $C^0$ and $\overline{C}$ denote the interior and closure of $C$, respectively.
For a vector $x=(x_1, \ldots, x_n)^T\in \R^n$, \yuan{let} $\|x\|$ \yuan{be a generic notation for vector norm.  Specifically, inherit the notion of $\|x\|_p$ to denote $\ell_p$ norm from functional analysis.}
% $\|x\|_\infty = \max_{1\le i\le n} |x_i|$ and $\|x\|_1=\sum_i |x_i|$ for the $\ell_\infty$- and $\ell_1$-norm of $x$ respectively.
% When $n$ is fixed, all norms on vectors are equivalent.
Let $B_\infty(x, \epsilon)=\{y: \| x-y\|_\infty \le \epsilon\}$ be an $\epsilon$-neighborhood of $x$ \yuan{under $\ell_\infty$ metric}.
For an $n\times n$ matrix $J=(J_{i,j})$, let $\|J\|_\infty$ denote the matrix norm induced by the $\ell_\infty$-norm on vectors in $\R^n$, i.e.,
\[
\|J\|_\infty = \sup_{x\neq 0} \frac{ \|Jx\|_\infty }{\|x\|_\infty}
=\max_{1\le i\le n}\sum_{j=1}^n |J_{i,j}|,
\]
$\|J\|_{\max}=\max_{i,j}|J_{ij}|$,
and let $\|J\|$ be a \yuan{generic notion for} matrix norm.
% Define the matrix maximum norm: .
We use a superscript ``*" to \yuan{mark} the true parameter.
\yuan{But in early sections of this paper, without causing ambiguity, we might omit it when stating the model.}
% The notation $\sum_{j<i}$  is a shorthand for $\sum_{i=1}^n \sum_{j=1, j<i}^n$.
% Write the partial derivative of a function vector $H(\beta, \gamma)$ evaluated at $(\widehat{\beta}, \widehat{\gamma})$ as
% \[
% \frac{ \partial H(\widehat{\beta}, \widehat{\gamma}) }{ \partial \beta }=\frac{ \partial H(\beta, \gamma) }{ \partial \beta }\bigg |_{\beta=\widehat{\beta},\gamma=\widehat{\gamma}}
% \]

\yuan{Next, we set up regularity conditions for our main theorems.}
Assume $\mu(\cdot)$ \yuan{has continuous third derivative}.
% Write $\mu^\prime$, $\mu^{\prime\prime}$ and $\mu^{\prime\prime\prime}$ as the first, second and third derivative of $\mu(x)$ on $x\in \R$, respectively.
% {Let $\epsilon_{n1}$ and $\epsilon_{n2}$ be two small positive numbers that tend to zero as $n$ goes to infinity.}
Recall $\pi_{ij}=\beta_i+\beta_j+z_{ij}^T \gamma$ as defined in \eqref{definition-pi}.
\yuan{Suppose there exist $b_{n0}, b_{n1}, b_{n2}, b_{n3}>0$ such that}
\begin{subequations}
\begin{gather}
\min_{i,j} \mu^\prime(\pi_{ij}) \cdot \max_{i,j} \mu^\prime(\pi_{ij}) >0,
\label{ineq-mu-key0}
\\
\label{ineq-mu-keya}
b_{n0}\le \min_{i,j} |\mu^\prime(\pi_{ij})| \le \max_{i,j}|\mu^\prime(\pi_{ij})|\le b_{n1},  \\
\label{ineq-mu-keyb}
\max_{i,j}|\mu^{\prime\prime}(\pi_{ij})| \le b_{n2}, \\
\label{ineq-mu-keyc}
\max_{i,j}|\mu^{\prime\prime\prime}(\pi_{ij})| \le b_{n3}.
\end{gather}
\end{subequations}
\yuan{hold for all $\beta \in B_\infty(\beta^*, \epsilon_{n1}), \gamma\in B_\infty(\gamma^*, \epsilon_{n2})$, where $\epsilon_{n1},\epsilon_{n2}>0$ are two diminishing numbers as $n\to\infty$.}

Condition \eqref{ineq-mu-key0} is mild, requiring that  the derivative of the expectation function
$\mu(x)$ is positive for all $x\in \R$ or negative.
\yuan{The conditions \eqref{ineq-mu-keya}--\eqref{ineq-mu-keyc} may seem quite technical and abstract for readers.  To help with intuitive understanding, let us illustrate them using Example \ref{example-a} with a logistic link function.}
In this case, $\mu(x)=e^x/(1+e^x)$.
Straight calculations show
\[
\mu^\prime(x) = \frac{e^x}{ (1+e^x)^2 },~~  \mu^{\prime\prime}(x) = \frac{e^x(1-e^x)}{ (1+e^x)^3 },~~ \mu^{\prime\prime\prime}(x) = \frac{e^x(1-4e^x+e^{2x})}{ (1+e^x)^4 }.
\]
It is easy to verify that \yuan{$\max\big\{
    |\mu^\prime(x)|,
    |\mu^{\prime\prime}(x)|,
    |\mu^{\prime\prime\prime}(x)|
    \big\}\leq 1/4$, where we used}
\[
|\mu^{\prime\prime}(x)| \le \frac{e^x}{ (1+e^x)^2 }  \left|\frac{(1-e^x)}{ (1+e^x) }\right|
\quad\textrm{and}\quad
|\mu^{\prime\prime\prime}(x)| =
\frac{e^x}{ (1+e^x)^2 }  \left| \frac{(1-e^x)^2}{ (1+e^x)^2 } - \frac{2e^x}{ (1+e^x)^2 }  \right|.
\]
\yuan{Therefore, in this example, we can set $b_{n1}=b_{n2}=b_{n3}=1/4$}
and
\begin{equation}\label{bn0-logistic}
b_{n0} = \min_{ i,j} \frac{e^{\pi_{ij}}}{ (1+e^{\pi_{ij}})^2 } \ge  \frac{ e^{ 2\|\beta^*\|_\infty + \|\gamma^*\|_1 z_* + 2\epsilon_{n1}+p\epsilon_{n2} } }{ (1 + e^{ 2\|\beta^*\|_\infty + \|\gamma^*\|_1 z_* + 2\epsilon_{n1}+p\epsilon_{n2} })^2 },
\end{equation}
where $z_*:=\max_{i,j} \| z_{ij} \|_\infty$.

\yuan{In addition to the regularity conditions \eqref{ineq-mu-key0}--\eqref{ineq-mu-keyc} this section, we shall also need the following assumptions:}
\begin{assumption}\label{assumption-a}
\yuan{(Bounded covariates) Suppose} $\max_{i,j} \|z_{ij}\|_\infty\leq C_z$ \yuan{holds for some universal constant $C_z$.}
\end{assumption}
\begin{assumption}\label{assumption-b}
\yuan{(Sub-exponential edge distribution) The distribution of} $a_{ij}- \E a_{ij}$ is sub-exponential, with parameter $h_{ij}$.  Denote $h_n:=\max_{i,j} h_{ij}$.
\end{assumption}

\yuan{Assumption \ref{assumption-a} is naturally satisfied by some popular dissimilarity measures between nodal covariates, such as Hamming distance.  If the observed $z_{ij}$'s or nodal covariates seem to vary wildly, we can simply apply a transformation such as sigmoid or probit functions to tame them into universally bounded edge covariates, see Section 6 of \citet{zhang2020edgeworth}.  Assumption \ref{assumption-b} is satisfied by many popular edge distributions, such as those in \citet{Yan:Qin:Wang:2015}.  We make this assumption mostly to make our narration succinct -- it can be replaced by any other conditions that guarantee
$|d_i - \E d_i|=O_p(n^{1/2})$ and $|\sum_{i<j} (a_{ij} - \E a_{ij})|=O_p(n)$, respectively.}

\subsection{Consistency}
\yuan{The asymptotic behavior of the estimator $(\hat\beta,\hat\gamma)$ critically depends on the curvature of $\mu(\beta,\gamma)$.
To study this curvature, we start with setting up some notation.  Define
\begin{equation}\label{eqn:def:F}
 F_i(\beta, \gamma)= \sum\limits_{j=1, j\neq i}^n \mu_{ij}(\beta, \gamma)-d_i,~~i=1, \ldots, n.
\end{equation}
For simplicity, we write $F(\beta, \gamma)=(F_1(\beta, \gamma), \ldots, F_n(\beta, \gamma))^T$.  Also denote $F_{\gamma,i}(\beta)$ to be $F_i(\beta, \gamma)$ for an arbitrary given $\gamma$, denote $F_\gamma(\beta)=(F_{\gamma,1}(\beta), \ldots, F_{\gamma,n}(\beta))^T$ and define $\hat\beta_\gamma$ to be the solution to $F_\gamma(\beta)=0$.  Set
\begin{eqnarray}
    \label{definition-Q}
    Q(\beta, \gamma)= \sum_{i<j} z_{ij} ( \mu_{ij}(\beta, \gamma) - a_{ij} ), \\
    \label{definition-Qc}
    Q_c(\gamma)= \sum_{i<j} z_{ij} ( \mu_{ij}(\widehat{\beta}_\gamma, \gamma) - a_{ij} ).
\end{eqnarray}
By definition, we have the following relationships:
\begin{equation*}\label{equation:FQ}
F(\widehat{\beta}, \widehat{\gamma})=0,~~F_\gamma(\widehat{\beta}_\gamma)=0,~~Q(\widehat{\beta}, \widehat{\gamma})=0,~~Q_c(\widehat{\gamma})=0.
\end{equation*}
Similar to \citet{Chatterjee:Diaconis:Sly:2011, Yan:Leng:Zhu:2016}, we define a notion called ``${\cal L}$ class matrices'' for narration convenience.  Given some $M\ge m>0$, we say an $n\times n$ matrix $V=(v_{ij})$ belongs to the matrix class $\mathcal{L}_{n}(m, M)$ if
$V$ is a diagonally balanced matrix with positive elements bounded by $m$ and $M$, i.e.,
\begin{equation}\label{eq1}
v_{ii}=\sum_{j=1, j\neq i}^{n} v_{ij}, ~~i=1,\ldots, n, ~~
m\le v_{ij} \le M, ~~ i,j=1,\ldots,n; i\neq j.
\end{equation}
Since $\pi$ is linear in $\beta$, for any $1\leq \{i\neq j\}\leq n$, we have
\begin{equation}
    \label{eqn::temp-1}
    \frac{\partial F_i(\beta, \gamma) }{\partial \beta_i } = \sum_{j\neq i} \mu^\prime(\pi_{ij}),\qquad
    \frac{\partial F_i(\beta, \gamma) }{\partial \beta_j } = \mu^\prime(\pi_{ij}).
\end{equation}
It is easy to verify that \eqref{eqn::temp-1} yields that when $\mu^\prime(x)>0$, $\beta\in B_\infty(\beta^*, \epsilon_{n1})$ and $\gamma\in B_\infty(\gamma^*, \epsilon_{n2})$, we have $F'_\gamma(\beta)\in \mathcal{L}_n(b_{n0}, b_{n1})$.
For simplicity, we assume $F'_\gamma(\beta)\in \mathcal{L}(b_{n0}, b_{n1})$ hereafter (if $-F'_\gamma(\beta)\in \mathcal{L}(b_{n0}, b_{n1})$, we could rewrite $\tilde{F}_\gamma(\beta) :=  - F_\gamma(\beta)$).  Define a convenient shorthand
$$
    V(\beta, \gamma): = F'_\gamma(\beta),
$$
and define an abbreviation $V=V(\beta^*, \gamma^*)$.
We will establish the consistency of the estimator $\hat{\beta}_\gamma$ using the theorems of Newton method, for which we shall need an explicity formulation of $F'_\gamma( \beta )$.
This inverse does not have a closed form, but fortunately, by mimicking \citet{simons-yao1999}, \citet{Yan:Zhao:Qin:2015} proposed a convenient approximate inversion formula $V\in \mathcal{L}_{n}(m, M)$ by
\begin{equation}\label{definition-s}
    S=\mathrm{diag}(1/v_{11}, \ldots, 1/v_{nn})
\end{equation}
 at approximate error of $O\left( M^2/(n^2 m^3) \right)$ under the matrix maximum norm (i.e.,
the maximum of all absolute elements of a matrix).

With the above notation preparations, now we commence the asymptotic analysis of the estimator.  First, we have
\begin{align}
    \frac{ \partial F_\gamma(\widehat{\beta}_\gamma) }{\partial \gamma^T}
    &=~
    \frac{ \partial F(\widehat{\beta}_\gamma, \gamma) }{\partial \beta^T}
    \frac{\partial \widehat{\beta}_\gamma }{\gamma^T} + \frac{\partial F(\widehat{\beta}_\gamma, \gamma)}{\partial \gamma^T} = 0,
    \label{equ-derivation-a}
    \\
    \frac{ \partial Q_c(\gamma)}{ \partial \gamma^T} &=~
    \frac{\partial Q(\widehat{\beta}_\gamma, \gamma)}{\partial \beta^T}
     \frac{\partial \widehat{\beta}_\gamma }{\gamma^T} + \frac{ \partial Q(\widehat{\beta}_\gamma, \gamma) }{ \partial \gamma^T}.
     \label{equ-derivation-b}
\end{align}
Combining \eqref{equ-derivation-a} and \eqref{equ-derivation-b}, the Jacobian matrix $Q_c^\prime(\gamma)=\partial Q_c^\prime(\gamma)/\partial \gamma$ has the following formulation:
}
\begin{eqnarray}\label{equation:Qc-derivative}
\frac{ \partial Q_c(\gamma) }{ \partial \gamma^T }  =
\frac{ \partial Q(\widehat{\beta}_\gamma, \gamma) }{ \partial \gamma^T}
 - \frac{ \partial Q(\widehat{\beta}_\gamma, \gamma) }{\partial \beta^T}
 \left[\frac{\partial F(\widehat{\beta}_\gamma,\gamma)}{\partial \beta^T}  \right]^{-1}
\frac{\partial F(\widehat{\beta}_\gamma,\gamma)}{\partial \gamma^T}.
\end{eqnarray}

The asymptotic behavior of $\widehat{\gamma}$ crucially depends on $Q_c^\prime(\gamma)$.
\yuan{But the $\widehat{\beta}_\gamma$ that appears in the definition of $ Q_c^\prime(\gamma)$ does not have a closed form.  To facilitate the quantitative study of the curvature of $Q_c(\gamma)$, define}
\begin{equation}
\label{definition-H}
H(\beta, \gamma) = \frac{ \partial Q(\beta, \gamma) }{ \partial \gamma^T } - \frac{ \partial Q(\beta, \gamma) }{\partial \beta^T } \left[ \frac{\partial F(\beta, \gamma)}{\partial \beta^T } \right]^{-1}
\frac{\partial F(\beta, \gamma)}{\partial \gamma^T },
\end{equation}
which \yuan{can be viewed as a relaxed version of} $\partial Q_c(\gamma) / \partial \gamma$.
When $\beta\in B_\infty(\beta^*, \epsilon_{n1})$, \yuan{by Section 10 of Supplemental Material,} we have
\begin{equation}\label{equation-H-appro}
\frac{1}{n^2} \big(H(\beta, \gamma^*)\big)_{ij} = \frac{1}{n^2} \big(H(\beta^*, \gamma^*)\big)_{ij} + o(1),
\end{equation}
\yuan{for each given $(i,j): 1\leq i<j\leq n$, where recall that the entries of $H$ are sums $n(n-1)/2$ of terms, thus $n^{-2}$ would be a proper rescaling factor.}
% \commentyuan{[Yuan says: it's unclear what \eqref{equation-H-appro} means -- which matrix norm does the remainder term's bound refer to??]}{\color{wqpcolor}It does not easy to prove the matrix norm.}

We assume $H(\beta, \gamma)$ is positively definite. \yuan{When $a_{ij}$ belongs to
exponential family of distributions, $H(\beta, \gamma)$ is the Fisher information matrix of the concentrated likelihood function on $\gamma$ (e.g. page 126 of \citet{Amemiya:1985}) and is thus positive definite.
See also Section \ref{section:app}.}
In fact, the asymptotic variance of $\widehat{\gamma}$ is $H^{-1}(\beta, \gamma)$ when $f(\cdot)$ in \eqref{model-2} \yuan{is an} exponential-family \yuan{distribution}; see the applications in Section \ref{section:app}. Thus, the asymptotic behavior of $\widehat{\gamma}$ will be ill-posed without this assumption.
Define
\begin{equation}
    \kappa_n
    :=
    \sup_{\beta\in B_\infty(\beta^*, \epsilon_{n1})} \| n^2 \cdot H^{-1}(\beta, \gamma^*)\|_\infty
\end{equation}

Now we formally state the consistency result.

\begin{theorem}\label{Theorem:con}
Let  $\sigma_n^2 = n^2 \| (V^{-1}-S) \mathrm{Cov}(F(\beta^*, \gamma^*)) (V^{-1}-S) \|_{\max}$.
Suppose Assumptions \ref{assumption-a} and \ref{assumption-b} and conditions \eqref{ineq-mu-key0}--\eqref{ineq-mu-keyc} hold, and
\begin{equation}\label{eq-theorema-ca}
 \frac{ \kappa_n^2 b_{n1}^4 b_{n2}  }{ b_{n0}^{3} }\left(\frac{b_{n2}h_n^2}{b_{n0}^{3}}+ \sigma_n\right)=o\left(\frac{n}{\log n}\right).
\end{equation}
Then the moment estimator $\widehat{\gamma}$ exists with \yuan{high probability, and we further have}
\begin{align*}\label{Newton-convergence-rate}
\| \widehat{\gamma} - \gamma^{*} \|_\infty &=  O_p\left(
 \frac{\kappa_n b_{n1} \log n }{ n }\left(\frac{h_n^2b_{n2}}{b_{n0}^{3}}+ \sigma_n\right) \right  )=o_p(1) \\
\| \widehat{\beta} - \beta^* \|_\infty &= O_p\left( \frac{ h_n }{b_{n0}}\sqrt{\frac{\log n}{n}} \right)=o_p(1).
\end{align*}
\end{theorem}

\yuan{Our proof of Theorem \ref{Theorem:con} analyzes a two-stage Newton method and is thus} different from \citet{Graham2017} \yuan{that uses a convergence rate analysis of the fixed point method} in \citet{Chatterjee:Diaconis:Sly:2011}.

\yuan{When} $f(\cdot)$ in model \eqref{model-2} is an exponential family distribution, then $V=\mathrm{Cov}(F(\beta^*, \gamma^*))$.
In this case, \yuan{the expression inside the norm of $\sigma_n^2$ simplifies into}
\[
 (V^{-1}-S) V (V^{-1}-S) = V^{-1} - S + \frac{ v_{ij}(1-\delta_{ij}) }{ v_{ii} v_{jj} },
\]
By Lemma \ref{pro:inverse:appro}, $\|V^{-1} - S\|_{\max} = O( b_{n1}^2 b_{n0}^{-3} n^{-2})$.
Thus, $\sigma_n^2 = O( b_{n1}^2/ b_{n0}^3)$. We have the following corollary.

\begin{corollary}
Assume $V=\mathrm{Cov}(F(\beta^*, \gamma^*))$ \yuan{and the conditions of Theorem \ref{Theorem:con} hold}.
If
\[
\frac{ \kappa_n^2 h_n^2 b_{n1}^5 b_{n2} }{ b_{n0}^{6} }=o\left( \frac{n}{\log n} \right),
\]
then
\begin{align*}\label{Newton-convergence-rate2}
\| \widehat{\gamma} - \gamma^{*} \|_\infty &=  O_p\left(
 \frac{\kappa_n b_{n1}^2 h_n^2b_{n2} \log n }{ n b_{n0}^{3}} \right  )=o_p(1) \\
\| \widehat{\beta} - \beta^* \|_\infty &= O_p\left( \frac{ h_n }{b_{n0}}\sqrt{\frac{\log n}{n}} \right)=o_p(1).
\end{align*}
\end{corollary}

When $f(\cdot)$ in model \eqref{model-2} belongs to exponential-family distributions and $\|\beta^*\|_\infty$ and $\|\gamma^*\|_\infty$ are \yuan{universally bounded}, then
$b_{n0}, b_{n1}, b_{n2}$ and $\sigma_n$ are constants. Further, if all covariates are bounded,
$H(\beta^*, \gamma^*)/n^2$ is approximately a constant matrix such that $\kappa_n$ is also a constant. In this case,
the conditions in Theorem \ref{Theorem:con} \yuan{easily hold}.
\yuan{Further, if} $b_{n0}, b_{n1}, b_{n2}, \kappa_n, h_n$ are constants, then the convergence rates of $\widehat{\beta}$ and $\widehat{\gamma}$ are $O_p( (\log n/n)^{1/2})$ and $O_p(\log n/n)$, \yuan{respectively}.
\yuan{This reproduces the error bound in} \citet{Chatterjee:Diaconis:Sly:2011}.
This convergence rate matches the minimax optimal upper bound $\|\widehat{\beta} - \beta\|_\infty =
O_p((\log p/n)^{1/2})$ for the Lasso estimator in the linear model with
a $p$-dimensional parameter vector $\beta$ and the
sample size $n$ \citep{lounici2008sup-norm}.
The convergence rate $O_p(\log n/n)$ for $\widehat{\gamma}$ is very close to the square root rate $N^{-1/2}$ in the classical large sample theory,
where $N= n(n-1)/2$.

\subsection{Asymptotic \yuan{normality} of $\widehat{\beta}$}

We derive the asymptotic expansion format of $\widehat{\beta}$ by applying a second order
Taylor expansion to $F(\widehat{\beta}, \widehat{\gamma})$ and showing {that various remainder terms are asymptotically negligible}.
% \commentyuan{[Yuan says:  I didn't understand this sentence] The fast convergence rate $1/n$ of $\widehat{\gamma}$ makes that the asymptotic representation of $\widehat{\beta}$
% does not depend on $\widehat{\beta}$.}

\begin{theorem}\label{Theorem-central-a}
Assume the conditions of Theorem \ref{Theorem:con} hold.
If
\[
 \kappa_n^2 b_{n1}^2 \left( \frac{ h_{n}^2 b_{n2} }{b_{n0}^3} + \sigma_n^2 \right)^2 =o\left( \frac{n}{\log n} \right),
\]
then for any fixed $i$,
\begin{equation*}
\widehat{\beta}_i- \beta_i=  v_{ii}^{-1}(d_i - \E d_i)  +  O_p\left( \frac{\kappa_n b_{n1}\log n}{ nb_{n0}} \Big(\frac{b_{n2}h_n^2}{b_{n0}^{3}}+ \sigma_n\Big)\right).
\end{equation*}
\end{theorem}

% \commentyuan{[Yuan says: I didn't understand this sentence] We use Lyapunov's conditions to illustrate the  \color{red}{correctness} of Theorem \ref{Theorem-central-a}.}
Let $u_{ii}= \sum_{j\neq i} \mathrm{Var}( a_{ij})$. If
$ \sum_{j\neq i} \E (a_{ij}-\E a_{ij})^3 / v_{ii}^{3/2}  \to 0$,
then, by the Lyapunov's central limit theorem,
$u_{ii}^{-1/2} \{d_i - \E(d_i)\}$ converges in distribution to the standard normal distribution.
When considering the asymptotic behaviors of the vector $(d_1, \ldots, d_r)$ with a fixed $r$, one could replace the degrees $d_1, \ldots, d_r$ by the independent random variables
$\tilde{d}_i=a_{i, r+1} + \cdots + a_{in}$, $i=1,\ldots,r$.
Therefore, we have the following lemma.

\begin{proposition}\label{pro:central:poisson}
\yuan{Under the conditions of Theorem \ref{Theorem-central-a}}, if $ u_{ii}^{-3/2} \sum_{j:j\neq i} \E (a_{ij}-\E a_{ij})^3  \to 0$, then we have: \\
(1)For any fixed $r\ge 1$,  $(d_1 - \E (d_1), \ldots, d_r - \E (d_r))$ are
asymptotically independent and normally distributed with \yuan{mean zero and marginal} variances $u_{11}, \ldots, u_{rr}$,
respectively. \\
(2)More generally, $\sum_{i=1}^n c_i(d_i-\E(d_i))/\sqrt{u_{ii}}$ is asymptotically normally distributed with mean zero
and variance $\sum_{i=1}^\infty c_i^2$ whenever $c_1, c_2, \ldots$ are fixed constants, and \yuan{$\sum_{i=1}^\infty c_i^2<\infty$}.
\end{proposition}

Part (2) follows from part (1) and the fact that
\begin{equation}
    \lim_{r\to\infty} \limsup_{t\to\infty}
    \mathrm{Var}\left( \sum_{k=r+1}^n c_i \frac{ d_i - \E (d_i) }{\sqrt{u_{ii}}}\right)=0
    \label{eqn::temp-2}
\end{equation}
by Theorem 4.2 of \citet{Billingsley:1995}.
\yuan{To see \eqref{eqn::temp-2}}, it suffices to show that the eigenvalues of
the covariance matrix of $(d_i - \E (d_i))/u_{ii}^{1/2}$, $i=r+1, \ldots, n$ are bounded by 2 for all $r<n$, \yuan{which is implied} by the well-known Perron-Frobenius theorem:
if $A$ is a symmetric positive definite matrix with diagonal elements equaling to $1$, with nonnegative off-diagonal elements,
then its largest eigenvalue is less than $2$.
In view of Proposition \ref{pro:central:poisson}, we immediately have the following corollary.

\begin{corollary}
\label{coro-theorema-a}
Assume that conditions in Theorem \ref{Theorem-central-a} hold. If $ u_{ii}^{-3/2} \sum_{j\neq i} \E (a_{ij}-\E a_{ij})^3  \to 0$,
then for fixed $k$ the vector $( u_{11}^{-1/2}v_{11} (\widehat{\beta}_1 - \beta^*), \ldots, u_{kk}^{-1/2}v_{kk} (\widehat{\beta}_k - \beta^*_k)$
converges in distribution to the $k$-dimensional multivariate standard normal distribution.
\end{corollary}

\subsection{Asymptotic \yuan{normality} of $\widehat{\gamma}$}

Let $T_{ij} = e_i+e_j$, where $e_i\in\mathbb R^n$ is all zero except its $i$th element equals 1.
 Define
\begin{equation*}
\begin{array}{c}
V(\beta, \gamma)=\frac{ \partial F(\beta, \gamma) }{ \partial \beta^T }, ~~
V_{Q\beta}(\beta, \gamma) = \frac{ \partial Q(\beta, \gamma) }{ \partial \beta^T}, \\
s_{ij}(\beta, \gamma) = (a_{ij}-\E a_{ij}) ( z_{ij} - V_{Q\beta}(\beta, \gamma) [V(\beta,\gamma)]^{-1} T_{ij}).
\end{array}
\end{equation*}
When evaluating $H(\beta,\gamma)$, $Q(\beta, \gamma)$, $V(\beta, \gamma)$ and $V_{Q\beta}(\beta, \gamma)$ at their true values
$(\beta^*, \gamma^*)$, we omit the arguments $\beta^*, \gamma^*$, i.e., $V=V(\beta^*, \gamma^*)$.
\yuan{Recall we earlier defined} $N=n(n-1)$.  \yuan{Also define}
\[
\bar{H}=  \lim_{n\to\infty} \frac{1}{N} H( \beta^*, \gamma^*),
\]
where \yuan{we recall the definition of} $H(\beta, \gamma)$ \yuan{from} \eqref{definition-H}.
% Assuming the above limit exists \commentyuan{[Yuan says: need some explanations! e.g. under which special example, this hold?]},
We have

\begin{theorem}
\label{theorem-central-b}
Let $U=\mathrm{Var}(d)$.
% and $ e_k$ is an $n$-dimensional column vector with $k$th element $1$ and others zeros.
Assume the conditions in Theorem \ref{Theorem:con} hold.
If $b_{n3}h_n^3 b_{n0}^{-3} = o( n^{1/2}/(\log n)^{3/2})$,
then we have
\[
\sqrt{N}(\widehat{\gamma}- \gamma^*) = \bar{H}^{-1} B_* + \bar{H}^{-1} \times \frac{1}{\sqrt{N}}  \sum_{i< j}
s_{ij} (\beta^*, \gamma^*) + o_p(1),
\]
where
\begin{equation}\label{defintion-Bias}
B_*=\lim_{n\to\infty} \frac{1}{2\sqrt{N}} \sum_{k=1}^n \left[\frac{ \partial^2 Q(\beta^*, \gamma^*) }{ \partial \beta_k \partial \beta^T}
V^{-1} U V^{-1} e_k \right].
\end{equation}
\end{theorem}

Note that $s_{ij}(\beta, \gamma)$, $i<j$, are independent vectors.
By Lyapunov's central limit theorem, we have
\begin{proposition}\label{pro:th4-b}
Let $\lambda_{ij}= \mathrm{Var} (a_{ij})$ and $\tilde{z}_{ij}= z_{ij} - V_{Q\beta} V^{-1} T_{ij}$.
For any nonzero vector $c=(c_1, \ldots c_p)^T$, if
\begin{equation}\label{eq-lyyaponu-condition}
\frac{ \sum_{i<j } (c^T \tilde{z}_{ij} )^3 \lambda_{ij}^3 }{ [\sum_{i <j } (c^T \tilde{z}_{ij} )^2 \lambda_{ij} ]^{3/2} } = o(1),
\end{equation}
then $(c^T \Sigma c)^{-1/2}  \sum_{i< j}
\tilde{s}_{\gamma_{ij}} (\beta^*, \gamma^*)$ converges in distribution to the standard normal distribution, where
$\Sigma= \mathrm{Cov}( Q - V_{Q\beta} V^{-1} H) $.
\end{proposition}

In view of Proposition \ref{pro:th4-b} and Theorem \ref{theorem-central-b},
we immediately have
\begin{corollary}\label{coro-theorem-b}
Assume the conditions in Theorem \ref{theorem-central-b} and \eqref{eq-lyyaponu-condition} hold.
Then
\begin{equation}
    \sqrt{N}c^T (\widehat{\gamma} - \gamma)
    \stackrel{d}\to
    N\Big(
        \bar{H}^{-1}B_*,
        c^T \bar{H}^T \Sigma \bar{H} c
    \Big)
\end{equation}
\end{corollary}

\yuan{When the edge distribution \eqref{model-2} belongs to exponential family, we have $V=U$. Consequently,} $\partial F(\beta^*, \gamma^*)/\partial \beta=\mathrm{Var}(d)$, $B_*$ and $\Sigma$ can be simplified as follows:
\begin{equation}
\label{defintion-B-22}
B_* =  \frac{1}{\sqrt{N}} \sum_{k=1}^n \frac{ \sum_{j\neq k} z_{kj} \mu_{kj}^{\prime\prime} (\pi_{ij}^*) }{ \sum_{j\neq k}  \mu_{kj}^{\prime} (\pi_{ij}^*) },
\end{equation}
and
\[
\Sigma = \sum_{i<j} z_{ij}z_{ij}^T \mu_{ij}^\prime -
 \sum_{i=1}^n \frac{ (\sum_{j\neq i} z_{ij}\mu_{ij}^\prime)(\sum_{j\neq i} z_{ij}^T \mu_{ij}^\prime) }{v_{ii}}.
\]
Note that asymptotic normality of $\widehat{\gamma}$ contains a bias term and needs to be corrected when constructing confidence interval and hypothesis testing.
Here, we employ the analytical bias correction formula in \citet{Dzemski2019}:
$\widehat{\gamma}_{bc} = \widehat{\gamma}- N^{-1/2}H^{-1}(\widehat{\beta},\widehat{\gamma}) \hat{B}$,
where $\widehat{B}$ \yuan{is a plug-in estimator for} $B_*$ \yuan{using} $\widehat{\beta}$ and
$\widehat{\gamma}$.
Other bias-corrections \yuan{include} \citet{Graham2017} and \citet{FVW2016}.

\section{Applications}
\label{section:app}
In this section, we illustrate the theoretical result by two applications:
the logistic distribution and Poisson distribution for $f(\cdot)$.
{Moreover, any other distributions such as the geometric distribution that lead to the well-defined moment estimator
could also be used, besides the logistic distribution and the Poisson distribution.
}

%The former case is presented in the Supplementary Material.

\subsection{The logistic model}
\label{section-logistic}

We consider the generalized $\beta$-model in \citet{Graham2017}  with the logistic distribution:
\[
\P(a_{ij}=1) = \frac{  e^{\beta_i + \beta_j + z_{ij}^T \gamma  }}{ 1 + e^{ \beta_i + \beta_j + z_{ij}^T \gamma } }.
\]
\citet{Graham2017} derived the consistency and asymptotic normality of the restricted MLE.
The aim of this application is to show that these properties of the unrestricted MLE continue to hold.
In this model, the MLE is the same as the moment estimator.

{
The numbers involved with the conditions in theorems are as follows.
Because $a_{ij}$'s are Bernoulli random variables, they are sub-exponential with $h_n=1$.
The numbers $b_{n0}, b_{n1}, b_{n2}$ and $b_{n3}$ are as defined in \eqref{bn0-logistic} and the paragraph right above it.
The condition \eqref{eq-theorema-ca} in Theorem \ref{Theorem:con} becomes that
\begin{equation}\label{app-co-a}
\kappa_n^2 \omega_n^3  = o\left( \sqrt{\frac{n}{\log n}} \right),
\end{equation}
where $\omega_n = e^{ 2\|\beta^*\|_\infty + \|\gamma^*\|_\infty }$.
}
By Theorem \ref{Theorem:con}, we have the following corollary.

\begin{corollary}
If \eqref{app-co-a} holds, then
\[
\|\widehat{\gamma}-\gamma^*\|_\infty = O_p\left( \frac{\kappa_n \omega_n^3\log n}{ n} \right),
\quad
\|\widehat{\beta} - \beta^*\|_\infty = O_p\left( \omega_n \sqrt{\frac{\log n}{n}} \right).
\]
\end{corollary}

We discuss the condition and convergence rates related to the graph density.
The expectation of the graph density is
\[
\rho_n :=\frac{1}{N} \sum_{1\le i < j \le n} \E a_{ij} = \frac{1}{N} \sum_{1\le i< j \le n}
\frac{ e^{\beta_i + \beta_j + z_{ij}^T \gamma } }{ 1 + e^{\beta_i + \beta_j + z_{ij}^T \gamma } },
\]
where $N=n(n-1)/2$.
 \yuan{
 To see what is $\kappa_n$, let us consider the case of that $z_{ij}$ is
 one dimension. By using $S$ in \eqref{definition-s} to approximate $V^{-1}(\beta,\gamma)$, one can get
 \[
 H(\beta,\gamma)=\sum_{1\le i<j \le n} z_{ij}^2 \mu^\prime( \pi_{ij} ) -
 \sum_{i=1}^n \frac{1}{v_{ii}} \left(\sum_{j=1,j\neq i}^n z_{ij}\mu^\prime(\pi_{ij})\right)^2.
 \]
 In this case, $\kappa_n$ is approximately the inverse of $n^{-2}H(\beta^*, \gamma^*)$, which
 depends on the covariates, the configuration of parameters, and the derivative of the mean function $\mu(\cdot)$. Since the relationship between $(\kappa_n, b_{n0})$ and $\rho_n$ depends on the configuration of the parameters $\beta$ and $\gamma$, where recall the definition of $b_{n0}$  in \eqref{bn0-logistic},
it is not possible to express $\kappa_n$ and $b_{n0}$ as a function of $\rho_n$ for a general $\beta$ and $\gamma$.
Therefore, we consider one special case that $\beta_1 = \cdots = \beta_n \leq c $  for illustration, where $c$ is a constant, and assume that $z_{ij}$ is independently drawn from
the standard normality.
%We further  assume that $\gamma = \gamma_0$ is a constant for easy exposition.
In this case, by large sample theory, we have
\[
\frac{1}{N}\sum_{1\le i<j \le n} z_{ij}^2 \mu^\prime( \pi_{ij} ) \stackrel{p.}{\to}
\frac{e^{2\beta_1}}{(1+e^{2\beta_1})^2}, ~~
\frac{1}{n}\sum_{j=1,j\neq i}^n z_{ij}\mu^\prime(\pi_{ij}) \stackrel{p.}{\to} 0,
\]
such that $\kappa_n \asymp 1/\rho_n$, where $a_n \asymp b_n$ means
$c_1 a_n \le b_n \le c_2 a_n$ with two constants $c_1$ and $c_2$ for sufficiently large $n$.
Further, $b_{n0} =O(   \rho_n )$.
\iffalse
\[
\kappa_n/\rho_n  \to 1
\rho_n =\frac{ e^{2\beta_1 + O(1) } }{ 1 + e^{2\beta_1 + O(1)} },~~~~
\]
Further, we assume $z_{ij}$ is one dimension and let
$(z_{12}, z_{13}, \ldots, z_{1n}, \ldots, z_{n-1,n})=(1,0,1,0, \ldots)$. In this case, $\kappa_n=4( 1+
e^{2\beta_1+\gamma_0})^2 /
e^{2\beta_1 + \gamma_0} \asymp 4 \rho_n^{-1}$.
\fi
Then the condition in Corollary 1 becomes
\[
\frac{\rho_n}{ (\log n/n)^{1/8} }  \to \infty,
\]
and, the convergence rates are
\[
\| \widehat{\gamma} - \gamma^{*} \|_\infty =  O_p\left(
 \frac{  \log n }{ n \rho_{n}^{4}} \right  ),~~
\| \widehat{\beta} - \beta^* \|_\infty = O_p\left( \frac{ 1 }{\rho_{n}}\sqrt{\frac{\log n}{n}} \right).
\]
Here, estimation consistency requires a strong assumption $\rho_n \gg (n/\log n)^{1/8}$. It would be of interest to relax it.
}

Since $a_{ij}$'s ($j<i$) are independent, it is easy to show the central limit theorem for $d_i$ and $N^{-1/2}\sum_{j<i} \tilde{s}_{ij}(\beta, \gamma)$ as given in \citet{su-qian2018}
and \citet{Graham2017} respectively. So by Theorems \ref{Theorem-central-a} and \ref{theorem-central-b}, the central limit theorem holds for $\widehat{\beta}$ and $\widehat{\gamma}$.
See  \citet{su-qian2018} and \citet{Graham2017} for details.

\subsection{The Poisson model}

We now consider the Poisson model in Example \ref{example-b}.
\iffalse
{We consider nonnegative integer edge weighte,} i.e., $a_{ij}\in \{0, 1, \ldots\}$.
We assume that all edges are independently distributed as Poisson random variables, where
\[
\P(a_{ij}=k) = \frac{ \lambda_{ij}^k }{k!} e^{-\lambda_{ij}},
\]
and $\lambda_{ij}=  e^{z_{ij}^T \gamma + \beta_i + \beta_j }$.
\fi
Recall that the expectation of $a_{ij}$ is $\lambda_{ij}=  e^{z_{ij}^T \gamma + \beta_i + \beta_j }$.
In this case, $\mu(x)=e^x$.
{The likelihood function is
\[
\P(A) \propto \exp\left( \sum_{i=1}^n \beta_i d_i  + \sum_{1\le i<j \le n} a_{ij}(z_{ij}^T \gamma) \right).
\]
%We will carry out simulations under this model in next section.
It is a special case of the general exponential random graph model, where
$(d^T, \sum_{i<j} a_{ij}z_{ij}^T)^T$ is the sufficient statistic for the parameter vector $(\beta^T, \gamma^T)^T$.
Therefore, the maximum likelihood equations
are identical to the moment equations defined in \eqref{eq:moment-beta} and \eqref{eq:moment-gamma}.
}
\iffalse
\begin{equation}\label{eq:likelihood-binary}
\begin{array}{c}
d_i  =  \sum_{j\neq i}  e^{z_{ij}^T \gamma + \beta_i + \beta_j } ,~~~i=1,\ldots, n, \\
\sum_{j<i} z_{ij}a_{ij} = \sum_{j<i}   z_{ij}e^{z_{ij}^T \gamma + \beta_i + \beta_j },
\end{array}
\end{equation}
which are identical to the maximum likelihood equations.
\fi
%where the $\mu$ function is $\mu(x)=e^x$.
Define
\[
q_n := \sup_{\beta \in B_\infty(\beta^*, \epsilon_{n1}), \gamma\in B_\infty(\gamma^*, \epsilon_{n2}) }\max_{i,j} | \beta_i + \beta_j + z_{ij}^T \gamma |.
\]
{
So $b_{ni}$'s ($i=0, \ldots, 3$) in inequalities \eqref{ineq-mu-keya}, \eqref{ineq-mu-keyb} and \eqref{ineq-mu-keyc} are
\[
b_{n0} = e^{-q_n}, ~~ b_{n1}= e^{q_n}, ~~ b_{n2} = e^{q_n}, ~~ b_{n3} = e^{q_n}.
\]
Clearly, Poisson($\lambda$) is sub-exponential with parameter $c\lambda$,
where $c$ is a constant;}
see Example 4.6 in \citet{zhang2020concentration}. Thus, $h_n$ in Assumption \ref{assumption-a} is $ce^{2q_n}$.
%% in Condition \ref{condition-diff-a}.
%Similar to the lines of arguments for proving Lemma 8, we have $h_{n2}= e^{2q_n}/n^{1/2}$.
%Let $\lambda_{n}$ be the smallest eigenvalue of $\bar{H}(\beta^*, \gamma^*)$. Then Condition \ref{condition-H} holds with $h_{n3}=\lambda_n$.
By Theorem \ref{Theorem:con}, we have the following corollary.

\begin{corollary}
\label{coro-b}
If $\kappa_n e^{7q_n} = o( (n/\log n)^{1/2})$, then
then
\[
\|\widehat{\gamma}-\gamma^*\|_\infty =  O_p(\frac{ \kappa_n e^{8 q_n} \log n}{n})=
o_p(1),~~~ \|\widehat{\beta} - \beta^*\|_\infty = O_p( \frac{ e^{2q_n}(\log n)^{1/2}}{n^{1/2}})=o_p(1).
\]
\end{corollary}

{
We discuss the condition and convergence rates related to the average weight.
The expectation of the average weight is
\[
\lambda_n := \frac{1}{N} \sum_i \E d_i  = \frac{1}{n} \sum_i \sum_{j\neq i} e^{\beta_i + \beta_j +z_{ij}^T \gamma }.
\]
As in the first application, $b_{n0}$, $b_{n1}$ and $b_{n2}$ can not be represented as functions on $\lambda_n$ for general parameters $\beta$ and $\gamma$.
\yuan{To get some intuitive understandings, let us consider a simple special case where $\beta_1=\cdots = \beta_n <c $ with a constant $c$, $\gamma$ is a constant and
$z_{ij}$ independently follows from a symmetric continuous distribution with a bounded support and the unit variance.
In this case,
\[
\kappa_n \asymp \lambda_n^{-1}, ~~\lambda_n = e^{2\beta_1 + O(1) }, ~~ c_1 \lambda_n \le b_{n0}, b_{n1}, b_{n2} \le c_2 \lambda_n,~~h_n = c_3 \lambda_n^2,
\]
where $c_1$, $c_2$ and $c_3$ are positive constants.
Then, the condition in Theorem \ref{Theorem:con} becomes
\[
 \lambda_n  =o\left( \left(\frac{n}{\log n}\right)^{1/4} \right),
\]
and the convergence rates are
\[
\| \widehat{\gamma} - \gamma^{*} \|_\infty =  O_p\left(
 \frac{\lambda_{n}^2 \log n }{ n } \right  ), \qquad
\| \widehat{\beta} - \beta^* \|_\infty = O_p\left( \lambda_n \sqrt{\frac{\log n}{n}} \right).
\]
%We can see that the convergence rate for $\widehat{\beta}$ does not depend on the average weight $\lambda_n$ while that for $\widehat{\gamma}$ increases with $\lambda_n$.
}
}

Note that $d_i=\sum_{j\neq i}a_{ij}$ is a sum of $n-1$ independent Poisson random variables.
Since $v_{ij} = \E a_{ij} = \lambda_{ij}$, we have
\[
e^{-q_n} \le v_{ij}= e^{ \beta_i + \beta_j + z_{ij}^T \gamma }
\le e^{q_n},~~1\le i< j\le n.
 \]
%% Use LOTUS to show that for $X\sim Pois(\lambda)$ and any function g, $E(Xg(X))=\lambda E(g(X+1))$.
By using the Stein-Chen identity [\citet{Stein1972, chen1975}] for the Poisson distribution, it is easy to verify that
\begin{equation}\label{eqn:poisson:exp}
\E (a_{ij}^3) = \lambda_{ij}^3 + 3\lambda_{ij}^2 + \lambda_{ij}.
\end{equation}
It follows
\[
\frac{\sum_{j\neq i} \E (a_{ij}^3) }{ v_{ii}^{3/2} } \le \frac{ (n-1)e^{q_n} }{ (n-1)^{3/2} e^{-q_n} }
= O( \frac{ e^{4q_n} }{ n^{1/2} }).
\]
If $e^{4q_n}  = o( n^{1/2} )$, then the above expression goes to zero.
For any nonzero vector $c=(c_1, \ldots c_p)^T$, if
\begin{equation}\label{eqn:lemma5:a}
\frac{ \sum_{j<i } (c^T \tilde{z}_{ij} )^3 \lambda_{ij}^3 }{ [\sum_{j<i } (c^T \tilde{z}_{ij} )^2 \lambda_{ij} ]^{3/2} } = o(1),
\end{equation}
This verifies the condition \eqref{eq-lyyaponu-condition}.
Consequently, by Corollaries \ref{coro-theorema-a} and \ref{coro-theorem-b}, we have the following result.

\begin{corollary}\label{corollary:poisson:central}
If \eqref{eqn:lemma5:a} holds and
$
\lambda_n^2\kappa_n^6 e^{28q_n}   = o(n^{1/2}/(\log n)^{3/2}),
$
then:
(1) $N^{1/2} \overline{\Sigma}^{-1/2}(\hat{\gamma}-\gamma^*)$ converges in distribution to multivariate normal distribution with mean $\overline{\Sigma}^{-1/2}\bar{H}^{-1}B_*$  and covariance $I_p$,
where $I_p$ is the identity matrix, where $\bar{\Sigma}= N^{-1}\bar{H}^{-1} \tilde{\Sigma} \bar{H}^{-1}$; \\
(2) for a fixed $r$,  the vector  $(v_{11}^{1/2}( \hat{\beta}_1 - \beta_1^*), \ldots, v_{rr}^{1/2}( \hat{\beta}_r - \beta_r^*)$ converges in distribution to the $r$-dimensional standard normal distribution.
\end{corollary}

\yuan{
\subsection{The probit model}
\label{section:application:probit}

The two examples above are exponential family of distributions. Here, we pay attention to
the probit distribution, which is not exponential.
Let $\phi(x)=(2\pi)^{1/2}e^{-x^2/2}$ be the standard normal density function
and $\Phi(x)=\int_{-\infty}^{x} \phi(x) dx $ be its the distribution function.
The probit model assumes
\[
\P(a_{ij}=1) = \Phi\left( \frac{1}{\sigma}(\beta_i + \beta_j + z_{ij}^\top \gamma)  \right),
\]
where $\sigma$ is the standard derivation. Since the parameters are scale invariable, we simply set $\sigma=1$. Then,
\[
\mu^\prime(x) = \phi(x), ~~\mu^{\prime\prime} (x) = \frac{ x}{\sqrt{2\pi}} e^{-x^2/2}.
\]
Since $\phi(x)=(2\pi)^{1/2}e^{-x^2/2}$ is an decreasing function on $|x|$, we have when $|x| \le Q_n$,
\[
 \frac{1}{2\pi} e^{ -Q_n^2/2} \le \phi(x) \le \frac{1}{2\pi}.
\]
Let $h(x)=xe^{-x^2/2}$. Then $h^\prime (x)= (1-x^2)e^{-x^2/2}$.
Therefore, when $x\in (0,1)$, $h(x)$ is an increasing function on its argument $x$; when $x\in (1, \infty)$,
$h(x)$ is an decreasing function on $x$. As a result, $h(x)$ attains its maximum value at $x=1$ when $x>0$.
Since $h(x)$ is a symmetric function,  we have $|h(x)|\le e^{-1/2}\approx 0.6$.
So
\[
b_{n0} \asymp \frac{1}{2\pi} e^{ -(\max_{i,j}\pi_{ij}^*)^2/2},~~ b_{n1}=\frac{1}{2\pi}, ~~b_{n2}=  (2\pi e)^{-1/2}.
\]

We only consider conditions for consistency here and those for central limit theorem are similar and omitted.
By \eqref{equation:Qc-derivative}, it is not difficult to verify
\[
\sigma_n^2 = O( n^4 \| V^{-1} - S\|_{\max}^2 ) = O( \frac{ b_{n1}^4 }{ b_{n0}^6 } ).
\]
The parameter $h_n$ in Assumption \ref{assumption-b} for a bounded random variable is a constant.
In view of Theorem \ref{Theorem:con}, we have the following corollary.
\begin{corollary}\label{corollary:con}
If
\begin{equation*}
 \kappa_n e^{ 3(\max_{i,j}\pi_{ij}^*)^2 } =o\left(\frac{n}{\log n}\right),
\end{equation*}
then the moment estimator $(\widehat{\beta},\widehat{\gamma})$ exists with \yuan{high probability, and we further have}
\begin{equation*}
\| \widehat{\gamma} - \gamma^{*} \|_\infty =  O_p\left(
\kappa_n e^{ 3(\max_{i,j}\pi_{ij}^*)^2/2 } \frac{\log n}{n} \right  ),~~
\| \widehat{\beta} - \beta^* \|_\infty = O_p\left( e^{ (\max_{i,j}\pi_{ij}^*)^2/2} \sqrt{\frac{\log n}{n}} \right).
\end{equation*}
\end{corollary}
}

\section{Discussion}
\label{section:sd}
{In this paper, we present a moment estimation for inferring the degree parameter $\beta$ and homophily parameter $\gamma$ in model \eqref{model-2}.}
We establish consistency of the moment estimator $(\widehat{\beta}, \widehat{\gamma})$
under several conditions and also derive its asymptotic \yuan{normality}.
{The convergence rates of $\widehat{\beta}$ and $\widehat{\gamma}$ are nearly optimal when
all parameters are bounded by a constant; but may not be optimal when the numbers $b_{n0}$, $b_{n1}$, $b_{n2}$ and $\kappa_n$ diverge.
}
Theorems \ref{Theorem-central-a} and \ref{theorem-central-b} require stronger assumptions than consistency, \yuan{but this is a widely-observed phenomenon in existing literature \citep{Yan:Leng:Zhu:2016, Yan-Jiang-Fienberg-Leng2018, zhang2021L2}.}
% Note that the asymptotic behavior of the MLE depends not only on $b_{n0}$ and $b_{n1}$, but also on the configuration of the parameters.
\yuan{Whether it is possible to establish consistency and asymptotic normality under even weaker conditions will be an interesting future work.}

\yuan{For cleanness, in this work, }we assume that $\max_{i,j} \| z_{ij} \|_\infty<c$ \yuan{is universally bounded.  In fact, our theory can be extended to allow it to slowly diverge.  It is another interesting future research to investigate how fast it can diverge while preserving consistency.}
% However, our main results can be extended to a case that allows $\max_{i,j} \| z_{ij} \|_\infty$ to increase with a slow rate.
% What can be said when some of $\|z_{ij}\|_\infty$'s are large? For
% example, some of the dyad covariates may increase with a fast rate.
% If the proportion of large covariates is bounded, then this will have little effect
% on the moment estimators. It is of interest to investigate whether asymptotic properties of the moment estimator continue to hold
%  when a large number of covariates diverge.

{The independent edge assumption leads to convenient characterization of the orders of $\| d - \E d\|_\infty$ and $\|\sum_{i<j} z_{ij}(a_{ij} - \E a_{ij}) \|_\infty$,
based on which, we
establish the central limit theorems of $d$ and $\sum_{i<j} z_{ij}a_{ij}$.
\yuan{For sub-exponential} $a_{ij}$, the orders of $\| d - \E d\|_\infty$ and $\|\sum_{i<j} z_{ij}(a_{ij} - \E a_{ij}) \|_\infty$ are $O( (n\log n)^{1/2})$ and $O( n\log n)$, respectively, up to a factor \yuan{determined by} the sub-exponential parameter $h_n$.
\yuan{Going forward, we can introduce slight dependency between edges.  Under such setting, we can still use some} Hoeffding-type inequalities for dependent random variables to establish tail bounds similar to those in this paper \citep{Delyon:2009}, \yuan{as long as edge dependency is sufficiently light}.
Remarkably, our \yuan{method-of-moments} estimation \yuan{remains effective, since it only requires specification of the marginal distributions of $a_{ij}$'s, not the \yuan{joint} distribution \yuan{of} $A$.  Certainly, quantitative study along this direction would require highly-nontrivial future efforts.}

\yuan{
Computation for covariate-assisted $\beta$-models is challenging in general.  The GLM package we use, which was also employed by \citet{chen2021analysis,stein2020sparse, stein2021sparse}, do not scale well.  Directly programming the Newton method seems more promising, but still might encounter difficulty when the network is $\Omega(10^5)$.  Unfortunately, the reduction method invented by \citet{zhang2021L2} only works for the classical and some generalized $\beta$-models without covariates, not for covariate-assisted $\beta$-models.  Exploring efficient computational methods is an interesting open challenge for future research.
}
}

\section*{Acknowledgements}

We are grateful to Editor, Associate Editor and two anonymous referees for their insightful comments and suggestions.
TY was partially supported by the National Natural Science Foundation of China (No. 11771171) and the Fundamental Research Funds for the Central Universities.

\section{Appendix}
\label{section:appendix}
\subsection{Preliminaries}

In this section, we present three results that will be used in the proofs.
The first is on the approximation error of using $S$ to approximate the inverse of $V$ belonging to the matrix class
$\mathcal{L}_n(b_{n0}, b_{n1})$,
where $V=(v_{ij})_{n\times n}$ and $S=\mathrm{diag}(1/v_{11}, \ldots, 1/v_{nn})$.
\citet{Yan:Zhao:Qin:2015} obtained the upper bound of the approximation error,
which has an order $n^{-2}$.
\citet{hillar2012inverses} gave a tight bound of $\| V^{-1} \|_\infty$.
These results are stated below as lemmas.

\begin{lemma}[Proposition 1 in \citet{Yan:Zhao:Qin:2015}] \label{pro:inverse:appro}
If $V \in \mathcal{L}_n(b_{n0}, b_{n1})$, then the following holds: %% for $n\ge 3$,
\begin{equation}\label{O-upperbound}
\|V^{-1} - S \|_{\max}=O\left(\frac{b_{n1}^2}{n^2b_{n0}^3}\right).
%% \le  \frac{ M(nM+(n-2)m)}{2m^3(n-2)(n-1)^2}+\frac{1}{2m(n-1)^2}+\frac{1}{m n(n-1)}
\end{equation}
\end{lemma}

\begin{lemma}[\citet{hillar2012inverses}]
\label{lemma-tight-V}
For $V\in \mathcal{L}_n(b_{n0}, b_{n1})$, we have
\[
\frac{1}{2b_{n1}(n-1)} \le \|V^{-1}\|_\infty \le \frac{3n-4}{2b_{n0}(n-1)(n-2)}.
\]
\end{lemma}

Let $F(x): \R^n \to \R^n$ be a function vector on $x\in\R^n$. We say that a Jacobian matrix $F^\prime(x)$ with $x\in \R^n$ is Lipschitz continuous on a convex set $D\subset\R^n$ if
for any $x,y\in D$, there exists a constant $\lambda>0$ such that
for any vector $v\in \R^n$ the inequality
\begin{equation*}
\| [F^\prime (x)] v - [F^\prime (y)] v \|_\infty \le \lambda \| x - y \|_\infty \|v\|_\infty
\end{equation*}
holds.
We will use the Newton iterative sequence to establish the existence and consistency of the moment estimator.
\citet{Gragg:Tapia:1974} gave the optimal error bound for the Newton method under the Kantovorich conditions
[\citet{Kantorovich1948Functional}].

\begin{lemma}[\citet{Gragg:Tapia:1974}]\label{pro:Newton:Kantovorich}
Let $D$ be an open convex set of $\R^n$ and $F:D \to \R^n$ a differential function
with a Jacobian $F^\prime(x)$ that is Lipschitz continuous on $D$ with Lipschitz coefficient $\lambda$.
Assume that $x_0 \in D$ is such that $[ F^\prime (x_0) ]^{-1} $ exists,
\begin{eqnarray*}
\| [ F^\prime (x_0 ) ]^{-1} \|_\infty  \le \aleph,~~ \| [ F^\prime (x_0) ]^{-1} F(x_0) \|_\infty \le \delta, ~~ \rho= 2 \aleph \lambda \delta \le 1,
\\
B_\infty(x_0, t^*) \subset D, ~~ t^* = \frac{2}{\rho} ( 1 - \sqrt{1-\rho} ) \delta = \frac{ 2\delta }{ 1 + \sqrt{1-\rho} }\le 2\delta.
\end{eqnarray*}
Then: (1) The Newton iterations $x_{k+1} = x_k - [ F^\prime (x_k) ]^{-1} F(x_k)$ exist and $x_k \in B_\infty(x_0, t^*) \subset D$ for $k \ge 0$. (2)
$x^* = \lim x_k$ exists, $x^* \in \overline{ B_\infty(x_0, t^*) } \subset D$ and $F(x^*)=0$.
\end{lemma}

\subsection{Error bound between $\widehat{\beta}_\gamma$ and $\beta^*$}

The lemma below shows that $F_{\gamma}(\beta)$ is Lipschitz continuous.
The proofs of all the lemmas in this section are given in the supplementary material.

\begin{lemma}\label{lemma:lipschitz-c}
Let $D=B_\infty(\beta^*, \epsilon_{n1}) (\subset \R^{n})$ be an open convex set containing the true point $\beta^*$.
For $\gamma\in B_\infty(\gamma^*, \epsilon_{n2})$, if inequality \eqref{ineq-mu-keyc} holds, then
the Jacobian matrix $F'_\gamma( x )$ of $F_\gamma(x)$ on $x$ is Lipschitz continuous on $D$ with the Lipschitz coefficient  $4b_{n2}(n-1)$.
\end{lemma}

Since $a_{ij}$, $1\le i<j \le n$, are independent and sub-exponential with parameters $h_{ij} (\le h_n)$,
by the concentration inequality for sub-exponential random variables [e.g., Corollary 5.17 in \citet{vershynin_2012} (\citeyear{vershynin_2012})], we have the following lemma.

\begin{lemma}\label{lemma-diff-F-Q}
With probability at least $1-O(n^{-1})$, we have
\begin{equation}
\| F(\beta^*, \gamma^*) \|_\infty = O( h_n \sqrt{n\log n} ), ~~
\| Q(\beta^*, \gamma^*) \|_\infty = O( h_n n \log n ).
\end{equation}
\end{lemma}

In view of Lemmas \ref{lemma:lipschitz-c} and \ref{lemma-diff-F-Q}, we obtain
the upper bound of the error between $\widehat{\beta}_\gamma$ and $\beta^*$  by using the Newton method.

\begin{lemma}\label{lemma-a}
Let $\epsilon_{n1}$ be a positive number and $\epsilon_{n2}=o( b_{n0}^{-1} (\log n)^{1/2} n^{-1/2})$.
Assume that \eqref{ineq-mu-keya}, \eqref{ineq-mu-keyb} and \eqref{ineq-mu-keyc}  hold.
If
\begin{equation}\label{equation:lemma-a}
 \frac{ b_{n2} h_{n} }{b_{n0}^2}  =o\left( \sqrt{ \frac{n}{\log n} } \right),
\end{equation}
then with probability at least $1-O(n^{-1})$, for $\gamma\in B_\infty(\gamma^*, \epsilon_{n2})$, $\widehat{\beta}_\gamma$ exists and satisfies
\[
\| \widehat{\beta}_\gamma - \beta^* \|_\infty = O_p\left( \frac{ h_n}{b_{n0}}\sqrt{\frac{\log n}{n}} \right) = o_p(1).
\]
\end{lemma}

\subsection{Proofs for Theorem \ref{Theorem:con}}

To show Theorem \ref{Theorem:con}, we need three lemmas below.

\begin{lemma}\label{lemma-Q-Lip}
Let $D=B_\infty(\gamma^*, \epsilon_{n2}) (\subset \R^{p})$ be an open convex set containing the true point $\gamma^*$.
Assume that  \eqref{ineq-mu-keya}, \eqref{ineq-mu-keyb}, \eqref{ineq-mu-keyc} and \eqref{equation:lemma-a}  hold.
If $\| F(\beta^*, \gamma^*) \|_\infty = O( h_{n}(n\log n)^{1/2} )$, then
$ Q_c(\gamma)$ is Lipschitz continuous on $D$ with the Lipschitz coefficient  $n^2 b_{n2} b_{n1}^{3} b_{n0}^{-3}$.
\end{lemma}

\begin{lemma}
\label{lemma-asym-expansion-beta}
Write $\widehat{\beta}^*$ as $\widehat{\beta}_{\gamma^*}$
and $V=\partial F(\beta^*, \gamma^*)/\partial \beta^T$. $\widehat{\beta}^*$ has the following expansion:
\begin{equation}\label{result-lemma2}
\widehat{\beta}^* - \beta^* =  V^{-1} F(\beta^*, \gamma^*) +V^{-1}R,
\end{equation}
where $R=(R_1, \ldots, R_{n})^T$ is the remainder term and
\[
\left\| V^{-1}  R \right \|_\infty = O_p( \frac{b_{n2} h_n^2 \log n}{nb_{n0}^3}  ).
\]
\end{lemma}

\begin{lemma}
\label{lemma-order-Q-beta}
Let $\Omega=\mathrm{Cov}(F(\beta^*, \gamma^*))$.
Let $\sigma_n^2 = n^2 \| (V^{-1}-S) \Omega (V^{-1}-S) \|_{\max}$.
For any $\beta \in B_\infty( \beta^*, \epsilon_{n1})$ and $\gamma \in B_\infty( \gamma^*, \epsilon_{n2})$,  we have
\[
 \| \frac{\partial Q (\beta, \gamma)}{ \partial \beta^T }
  (\widehat{\beta}^*-\beta^*) \|_\infty = O_p\left(  nb_{n1} \log n (\frac{b_{n2}h_n^2}{b_{n0}^{3}}+ \sigma_n) \right).
\]
Further, when $\Omega=-\partial F(\beta^*, \gamma^*)/\partial \beta^T$,  we have
\[
 \| \frac{\partial Q (\beta, \gamma)}{ \partial \beta^T }
  (\widehat{\beta}^*-\beta^*) \|_\infty = O_p\left(\frac{h_n^2b_{n1}b_{n2} n\log n}{b_{n0}^{3}} \right).
\]
\end{lemma}

Now we are ready to prove  Theorem \ref{Theorem:con}.

\begin{proof}[Proof of Theorem \ref{Theorem:con}]
We construct the Newton iterative sequence to show the consistency.
In view of Lemma \ref{pro:Newton:Kantovorich},
it is sufficient to demonstrate the Newton-Kantovorich conditions.
We set $\gamma^*$ as the initial point $\gamma^{(0)}$ and $\gamma^{(k+1)}=\gamma^{(k)} - [Q_c^\prime(\gamma^{(k)})]^{-1}Q_c(\gamma^{(k)})$.

By Lemma \ref{lemma-a}, with probability at least $1-O(n^{-1})$, we have
we have
\[
\| \widehat{\beta}_\gamma - \beta^* \|_\infty = O_p\left( \frac{ h_n }{b_{n0} }\sqrt{\frac{\log n}{n}} \right). %%:=O_p(\epsilon_n).
\]
This shows that $\widehat{\beta}_{\gamma^{(0)}}$ exists such that $Q_c(\gamma^{(0)})$ and $Q_c^\prime(\gamma^{(0)})$ are well defined.
%This also shows that in every iterative step, $\gamma^{(k+1)}$ exists as long as $\gamma^{(k)}$ exists.

Recall the definition of $Q_c(\gamma)$ and $Q(\beta, \gamma)$ in \eqref{definition-Q} and \eqref{definition-Qc}.
By Lemmas \ref{lemma-diff-F-Q} and \ref{lemma-order-Q-beta}, we have
\begin{eqnarray*}
\|Q_c(\gamma^*)\|_\infty  & \le &  \|Q(\beta^*, \gamma^*)\|_\infty + \|Q(\widehat{\beta}_{\gamma^*}, \gamma^*)-Q(\beta^*, \gamma^*)\|_\infty\\
&=& O_p\left(  nb_{n1} \log n (\frac{h_n^2b_{n2}}{b_{n0}^{3}}+ \sigma_n)\right).
\end{eqnarray*}
By Lemma \ref{lemma-Q-Lip}, $\lambda=n^2  b_{n1}^{3} b_{n2} b_{n0}^{-3} $.
Note that
$\aleph=\| [Q_c^\prime(\gamma^*)]^{-1} \|_\infty = O ( \kappa_n n^{-2})$.
Thus,
\[
\delta = \| [Q_c^\prime(\gamma^*)]^{-1} Q_c(\gamma^*) \|_\infty =  O_p\left(
 \frac{\kappa_n b_{n1} \log n }{ n }(\frac{h_n^2b_{n2}}{b_{n0}^{3}}+ \sigma_n) \right  ).
\]
 As a result, if equation \eqref{eq-theorema-ca} holds, then
\[
\rho=2\aleph \lambda \delta =
O_p\left( \frac{ \kappa_n^2 b_{n1}^4 b_{n2} \log n }{n b_{n0}^{3} }(\frac{h_n^2b_{n2}}{b_{n0}^{3}}+ \sigma_n) \right)=o_p(1).
\]
By Theorem \ref{pro:Newton:Kantovorich},
the limiting point of the sequence $\{\gamma^{(k)}\}_{k=1}^\infty$ exists, denoted by $\widehat{\gamma}$, and satisfies
\[
\| \widehat{\gamma} - \gamma^* \|_\infty = O_p(\delta).
\]
By Lemma \ref{lemma-a}, $\widehat{\beta}_{\widehat{\gamma}}$ exists, denoted by $\widehat{\beta}$, and
$( \widehat{\beta}, \widehat{\gamma})$ is the moment estimator.
It completes the proof.
\end{proof}

\setlength{\itemsep}{-1.5pt}
\setlength{\bibsep}{0ex}
\bibliography{reference3}
\bibliographystyle{apalike}

\newpage

\begin{center}
\Large{Supplementary material for ``Asymptotic theory in network models with covariates and a growing number of node parameters"}
\end{center}
\iffalse
\author{Qiuping Wang\thanks{ Both authors  contributed equally. School of Mathematics and Statistics, Zhaoqing University, Zhaoqing, 526061, China.
\texttt{Email:} qp.wang@mails.ccnu.edu.cn.}
\hspace{5mm}
Yuan Zhang*\thanks{Department of Statistics, The Ohio State University, Columbus, 43210, U.S.A.
\texttt{Email:} yzhanghf@stat.osu.edu}
\hspace{5mm}
Ting Yan\thanks{Department of Statistics, Central China Normal University, Wuhan, 430079, China.
\texttt{Email:} tingyanty@mail.ccnu.edu.cn.}
\\
 Zhaoqing University\\
 The Ohio State University \\
 Central China Normal University
}
\fi

Section \ref{section-simulation} contains details of simulation studies and the application to the Enran email data.
The proofs of Lemmas 4, 5 and 6 are given in Sections \ref{section-lemma4}, \ref{section-lemma5} and \ref{section-lemma6}, respectively.
Sections \ref{section-lemma7}, \ref{section-lemma8} and \ref{section-lemma9} contain the proofs of Lemmas 7, 8 and 9, respectively.
We present the proofs of Theorems 2 and 3 in Sections \ref{section-theorem2} and \ref{section-theorem3}, respectively.
The proof of equation (15) is in Section \ref{section-proof413}.
Section \ref{section-simpli} contains the detailed simplification calculations of the bias term $B_*$ in equation (19).
%Section \ref{section6} contains one Lemma for checking that the Poisson random variables are sub-exponential.
%Section \ref{section-table} contains the table fitted in the Enran email dataset.
%Here, $\|V\|_\infty$ is the maximum absolute row sum of a matrix, which is the matrix norm induced by the infinity
%norm $\|x\|_\infty$ on vectors in $R^n$.
%this paper, In and 1n denote the nn identity matrix and the n-dimensional column vector consisting of all
%ones, respectively.
The following inequalities in the main text are restated here, which will be used in the proofs repeatedly.
\begin{subequations}
\begin{gather}
\label{ineq-mu-keya}
b_{n0}\le \min_{i,j} |\mu^\prime(\pi_{ij})| \le \max_{i,j}|\mu^\prime(\pi_{ij})|\le b_{n1},  \\
\label{ineq-mu-keyb}
\max_{i,j}|\mu^{\prime\prime}(\pi_{ij})| \le b_{n2}, \\
\label{ineq-mu-keyc}
\max_{i,j}|\mu^{\prime\prime\prime}(\pi_{ij})| \le b_{n3}.
\end{gather}
\end{subequations}

\section{Simulation studies}
\label{section-simulation}

We set the parameter values to be a linear form, i.e.,
$\alpha_{i}^* = (i-1)L/(n-1)$ for $i=1, \ldots, n$.
We considered four different values for $L$ as $L\in \{0 , \log(\log n), (\log n)^{1/2}, \log n \}$.
By allowing $\alpha^*$ to grow with $n$, we intended to assess the asymptotic properties under different asymptotic regimes.
Each node had two covariates $X_{i1}$ and $X_{i2}$. Specifically,
$X_{i1}$ took values positive one or negative one with equal probability and $X_{i2}$ came from a $Beta(2,2)$ distribution.
All covariates were independently generated.
%%Each element of the $p$-dimensional node-specific covariate $X_i$ is independently generated from a $Beta(2,2)$ distribution.
The edge-level covariate $z_{ij}$ between nodes $i$ and $j$ took the form: $z_{ij}=(x_{i1}*x_{j1}, |x_{i2}-x_{j2}|)^\top$.
For the homophily parameter, we set  $\gamma^*=(0.5, 1)^\top$.
Thus, the homophily effect of the network is determined by a weighted sum of the similarity measures of the two covariates between two nodes.

By Corollary 5, %%\ref{corollary:poisson:central},
given any pair $(i,j)$, $\hat{\xi}_{i,j} = [\hat{\beta}_i-\hat{\beta}_j-(\beta_i^*-\beta_j^*)]/(1/\hat{v}_{i,i}+1/\hat{v}_{j,j})^{1/2}$
converges in distribution to the standard normality, where $\hat{v}_{i,i}$ is the estimate of $v_{i,i}$
by replacing $(\beta^*, \gamma^*)$ with $(\widehat{\beta}, \widehat{\gamma})$.
Therefore, we assessed the asymptotic normality of $\hat{\xi}_{i,j}$ using the quantile-quantile (QQ) plot.
Further, we also recorded the coverage probability of the 95\% confidence interval and the length of the confidence interval.
The coverage probability and the length of the confidence interval of $\widehat{\gamma}$ were also reported.
Each simulation was repeated $10,000$ times.

We did simulations with network sizes $n=100$ and $n=200$ and found that the QQ-plots for these two network sizes were similar.
Therefore, we only show the QQ-plots for $n=100$ to save space.
Further, the QQ-plots for $L=0$ and $L=\log(\log n)$  are similar.
Also, for $L=(\log n)^{1/2}$ and $L=\log n$, they are similar.
Therefore we only show those for $L=\log (\log n)$ and $L=\log n$  in Figure \ref{figure-qq}.
In this figure, the horizontal and vertical axes are the theoretical and empirical quantiles, respectively,
and the straight lines correspond to the reference line $y=x$.
In Figure \ref{figure-qq}, when $L=\log (\log n)$, the empirical quantiles coincide well with the theoretical ones.
When $L=(\log n)^{1/2}$, the empirical quantiles have a little derivation from the theoretical ones in the upper tail of the right bottom subgraph.
These figures show that there may be large space for improvement on the growing rate of $\| \beta\|_\infty$ in the conditions in Corollary 5. %%\ref{corollary:poisson:central}.

\begin{figure}[!htb]
\centering
\includegraphics[ width=3.5in, height=3in, angle=0]{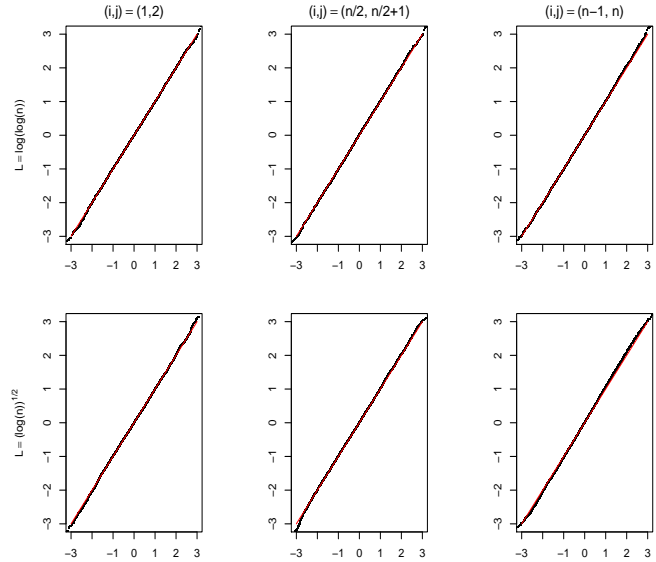}
\caption{The QQ plots of $\hat{\xi}_{i,j}$ (n=100). }
\label{figure-qq}
\end{figure}

Table \ref{Table:alpha} reports the coverage probability of the 95\% confidence interval for $\beta_i - \beta_j$ and the length of the confidence interval.
As we can see, the length of the confidence interval decreases as $n$ increases, which qualitatively agrees with the theory.
The coverage frequencies are all close to the nominal level $95\%$. On the other hand, the length of the confidence interval decreases as $L$ increases.
It seems  a little unreasonable. Actually,  the theoretical length of the $95\%$ confidence interval is $(1/v_{ii} + v_{jj})^{1/2}$ multiple by
a constant factor. Since $v_{ii}$ is a sum of a set of exponential items,  it becomes quickly larger as $L$ increases. As a result,
the length of confidence interval decreases as long as the estimates are close to the true values.  The simulated coverage probability results shows that
the estimates are very good.
So, this phenomenon that the length of confidence interval decreases in Table \ref{Table:alpha}, also agrees with the theory.

{\renewcommand{\arraystretch}{1}
\begin{table}[!h]\centering
\caption{The reported values are the coverage frequency ($\times 100\%$) for $\beta_i-\beta_j$ for a pair $(i,j)$ / the length of the confidence interval($\times 10$).}
\label{Table:alpha}
%\vskip5pt
\begin{spacing}{1}
\begin{tabular}{ccccccc}
\hline
n       &  $(i,j)$ & $L=0$ & $L=\log(\log n)$ & $L=(\log n)^{1/2}$ & $L=\log n$ \\
\hline
100         &$(1,2)   $&$94.56 / 4.60 $&$ 95.08 / 2.97 $&$ 94.80 / 2.42 $&$ 94.69 / 0.97  $ \\
            &$(50,51) $&$94.72 / 4.60 $&$ 94.93 / 2.04 $&$ 94.89 / 1.43 $&$ 94.83 / 0.31 $ \\
            &$(99,100)$&$95.12 / 4.60 $&$ 94.41 / 1.40 $&$ 94.38 / 0.85 $&$ 94.13 / 0.10$ \\
&&&&&&\\
200         &$(1,2)     $&$95.20 / 3.24 $&$ 94.79 / 2.01 $&$ 94.76 / 1.63 $&$ 95.09 / 0.52 $ \\
            &$(100,101) $&$95.03 / 3.24 $&$ 94.75 / 1.33 $&$ 94.91 / 0.92 $&$ 95.47 / 0.14 $ \\
            &$(199,200) $&$94.58 / 3.24 $&$ 95.05 / 0.88 $&$ 94.63 / 0.52 $&$ 93.90 / 0.04$ \\
\hline
\end{tabular}
\end{spacing}
\end{table}
}

Table \ref{Table:gamma} reports the coverage frequencies for the estimate $\widehat{\gamma}$
and bias corrected estimate  $\widehat{\gamma}_{bc}$ %% (= \widehat{\gamma}- \hat{\Sigma}^{-1} \hat{B}/\sqrt{n(n-1)})$
at the nominal level $95\%$,
and the standard error.
As we can see, the differences between the coverage frequencies with  uncorrected estimates and bias corrected estimates are very small.
All coverage frequencies are very close to the nominal level. The bias under the case of the Poisson distribution is very small in our simulation design.

\iffalse
different from the simulation results in the logistic model in \cite{Graham2017},
in which the coverage frequencies for the uncorrected estimate is far away the nominal level.
the coverage frequencies for the uncorrected estimate is visibly below the nominal level with at least $10$ percentage points
and the bias correction estimate dramatically improve the coverage frequencies, whose coverage frequencies are close to the nominal level when the MLE exists with a high frequency.
On the other hand, when $n$ is fixed, the average absolute bias of $\bs{\widehat{\gamma}}$ increases as $L$ becomes larger and so is the standard error.
\fi

{\renewcommand{\arraystretch}{1}
\begin{table}[!htbp]\centering
\caption{
The reported values are the coverage frequency ($\times 100\%$) for $\gamma_i$ for $i$ / length ($\times 10$) of confidence interval ($\bs{\gamma}^*=(0.5, 1)^\top$).
}
\label{Table:gamma}
%\vskip5pt
\begin{tabular}{cclllcc}
\hline
$n$     &   $\widehat{\gamma}$  & $L=0$ & $L=\log(\log n)$ & $L=(\log n)^{1/2}$ & $L=\log n$ \\
\hline
$100$   & $\hat{\gamma}_1$            &$95.13 / 0.52 $&$ 95.25 / 0.22 $&$ 94.92 / 0.15 $&$ 95.04 / 0.02 $ \\

        & $\hat{\gamma}_{bc, 1}$      &$95.11 / 0.52 $&$ 95.25 / 0.22 $&$ 94.92 / 0.15 $&$ 95.04 / 0.02 $ \\

        & $\hat{\gamma}_2$            &$94.98 / 3.08 $&$ 95.28 / 1.31 $&$ 95.00 / 0.88 $&$ 95.06 / 0.15  $ \\

        & $\hat{\gamma}_{bc, 2}$      &$94.93 / 3.08 $&$ 95.29 / 1.31 $&$ 95.02 / 0.88 $&$ 95.06 / 0.15  $\\

$200$   & $\hat{\gamma}_1$            &$94.87 / 0.26 $&$ 95.49 / 0.10 $&$ 95.07 / 0.07 $&$ 94.92 / 0.007 $ \\
        & $\hat{\gamma}_{bc, 1}$      &$94.87 / 0.26 $&$ 95.47 / 0.10 $&$ 95.08 / 0.07 $&$ 94.91 / 0.007 $\\

        & $\hat{\gamma}_2$            &$95.31 / 1.52 $&$ 95.12 / 0.59 $&$ 94.97 / 0.39 $&$ 94.49 / 0.041 $ \\
        & $\hat{\gamma}_{bc, 2}$      &$95.31 / 1.52 $&$ 95.12 / 0.59 $&$ 94.95 / 0.39 $&$ 94.49 / 0.041 $\\

\hline
\end{tabular}
\end{table}
}

\subsection{A real data example}
%%is valuable because it is one of the very few collections of organizational emails that are publicly available

%%Enron Email Dataset downloaded from : https://www.cs.cmu.edu/~enron/.
%%And it is the May 7, 2015 Version of dataset.

We use the Enron email dataset as an example analysis [\cite{Cohen2004}], available from \url{https://www.cs.cmu.edu/~enron/}.  %% (Cohen, 2009)
The Enron email data was acquired and made public by the Federal Energy Regulatory Commission during its investigation into fraudulent accounting practices.
The raw data is messy and needs to be cleaned before any analysis is conducted.
\cite{Zhou2007} applied data cleaning strategies to compile the Enron email dataset.
We use their cleaned data for the subsequent analysis.
The resulting data comprises $21,635$ messages sent between $156$ employees
with their covarites information. There are $6,650$ messages having more than one recipient across their `To', `CC' and
`BCC' fields, with a few messages having more than 50 recipients. For our analysis,
we exclude messages with more than ten recipients, which is a subjectively chosen cut-off that avoids
emails sent en masse to large groups. Each employee has three categorical variables: departments of these employees (Trading, Legal, Other),
the genders (Male, Female) and seniorities (Senior, Junior). Employees are labelled from $1$ to $156$.
The $3$-dimensional covariate vector $z_{ij}$ of edge $(i,j)$ is formed by using a homophilic matching function between these $3$ covariates of two employees $i$ and $j$, i.e.,
if $x_{ik}$ and $x_{jk}$ are equal, then $z_{ijk}=1$; otherwise $z_{ijk}=0$.

%%We define the covariate for each edge as a $3$-dimensional vector consisting of the matching function , where for categorical variables,
%%the difference is defined as the indicator whether they are equal. homophilic matching function.

For our analysis, we removed the employees  ``32" and ``37" with zero degrees, where the estimators of the corresponding node parameters do not exist.
This leaves a connected network with $154$ nodes.
The minimum, $1/4$ quantile, median, $3/4$ quantile and maximum values of $d$ are $1$, $95$, $220$, $631$ and $4637$, respectively.
It exhibits a strong degree heterogeneity. The estimators of $\alpha_i$ with their estimated standard errors are given in Table \ref{Table:alpha:real}.
%\ref{Table:alpha:real}.
The estimates of degree parameters vary widely:
from the minimum $-4.36$ to maximum $2.97$.
We then test three null hypotheses $\beta_2=\beta_3$, $\beta_{76}=\beta_{77}$ and $\beta_{151}=\beta_{154}$, using
the homogeneity test statistics $\hat{\xi}_{i,j} = |\hat{\beta}_i-\hat{\beta}_j|/(1/\hat{v}_{i,i}+1/\hat{v}_{j,j})^{1/2}$.
The obtained $p$-values turn out to be $1.7\times 10^{-24}$, $1.8\times 10^{-4}$ and $6.2\times 10^{-23}$, respectively,
confirming the need to assign one parameter to each node to characterize the heterogeneity of degrees.

{\renewcommand{\arraystretch}{1}
\begin{table}[!hbt]\centering
\scriptsize
%\vskip5pt
\caption{The estimates of $\beta_i$ and their standard errors in the Enron email dataset.}
\label{Table:alpha:real}
\begin{tabular}{cccc c cccc c cccc c cccc}
\hline
Node & $d_i$  &  $\hat{\beta}_i$ & $\hat{\sigma}_i$ & Node & $d_i$  &  $\hat{\beta}_i$ & $\hat{\sigma}_i$ &  Node & $d_i$  &  $\hat{\beta}_i$ & $\hat{\sigma}_i$
& Node & $d_i$  &  $\hat{\beta}_i$ & $\hat{\sigma}_i$ \\
\hline
1 &$ 723 $&$ 1.03 $&$ 0.37 $& 41 &$ 309 $&$ 0.15 $&$ 0.57 $& 79 &$ 309 $&$ -0.46 $&$ 0.79 $& 117 &$ 1176 $&$ 1.49 $&$ 0.29 $ \\
2 &$ 67 $&$ -1.36 $&$ 1.22 $& 42 &$ 281 $&$ 0.08 $&$ 0.6 $& 80 &$ 281 $&$ -0.08 $&$ 0.65 $& 118 &$ 398 $&$ 0.4 $&$ 0.5 $ \\
3 &$ 275 $&$ 0.03 $&$ 0.6 $& 43 &$ 690 $&$ 0.96 $&$ 0.38 $& 81 &$ 690 $&$ 0.32 $&$ 0.53 $& 119 &$ 369 $&$ 0.35 $&$ 0.52 $ \\
4 &$ 1202 $&$ 1.54 $&$ 0.29 $& 44 &$ 234 $&$ -0.13 $&$ 0.65 $& 82 &$ 234 $&$ 0.32 $&$ 0.52 $& 120 &$ 2673 $&$ 2.33 $&$ 0.19 $ \\
5 &$ 678 $&$ 0.94 $&$ 0.38 $& 45 &$ 704 $&$ 1 $&$ 0.38 $& 83 &$ 704 $&$ -1.45 $&$ 1.27 $& 121 &$ 571 $&$ 0.75 $&$ 0.42 $ \\
6 &$ 249 $&$ -0.07 $&$ 0.63 $& 46 &$ 952 $&$ 1.27 $&$ 0.32 $& 84 &$ 952 $&$ -0.74 $&$ 0.89 $& 122 &$ 2174 $&$ 2.15 $&$ 0.21 $ \\
7 &$ 375 $&$ 0.35 $&$ 0.52 $& 47 &$ 998 $&$ 1.38 $&$ 0.32 $& 85 &$ 998 $&$ 0.72 $&$ 0.43 $& 123 &$ 343 $&$ 0.26 $&$ 0.54 $ \\
8 &$ 40 $&$ -1.88 $&$ 1.58 $& 48 &$ 686 $&$ 0.99 $&$ 0.38 $& 86 &$ 686 $&$ -2.04 $&$ 1.71 $& 124 &$ 115 $&$ -0.8 $&$ 0.93 $ \\
9 &$ 428 $&$ 0.48 $&$ 0.48 $& 49 &$ 1224 $&$ 1.54 $&$ 0.29 $& 87 &$ 1224 $&$ -0.31 $&$ 0.71 $& 125 &$ 195 $&$ -0.29 $&$ 0.72 $ \\
10 &$ 95 $&$ -1.01 $&$ 1.03 $& 50 &$ 141 $&$ -0.63 $&$ 0.84 $& 88 &$ 141 $&$ -1.29 $&$ 1.16 $& 126 &$ 102 $&$ -0.96 $&$ 0.99 $ \\
11 &$ 231 $&$ -0.12 $&$ 0.66 $& 51 &$ 101 $&$ -0.95 $&$ 1 $& 89 &$ 101 $&$ -1.31 $&$ 1.17 $& 127 &$ 180 $&$ -0.4 $&$ 0.75 $ \\
12 &$ 31 $&$ -2.16 $&$ 1.8 $& 52 &$ 1 $&$ -5.57 $&$ 10 $& 90 &$ 1 $&$ 0.52 $&$ 0.48 $& 128 &$ 67 $&$ -1.39 $&$ 1.22 $ \\
13 &$ 85 $&$ -1.15 $&$ 1.08 $& 53 &$ 1138 $&$ 1.46 $&$ 0.3 $& 91 &$ 1138 $&$ 1.17 $&$ 0.35 $& 129 &$ 185 $&$ -0.38 $&$ 0.74 $ \\
14 &$ 53 $&$ -1.62 $&$ 1.37 $& 54 &$ 66 $&$ -1.41 $&$ 1.23 $& 92 &$ 66 $&$ 1.59 $&$ 0.28 $& 130 &$ 1798 $&$ 1.96 $&$ 0.24 $ \\
15 &$ 182 $&$ -0.36 $&$ 0.74 $& 55 &$ 155 $&$ -0.5 $&$ 0.8 $& 93 &$ 155 $&$ -1.02 $&$ 1.03 $& 131 &$ 3157 $&$ 2.5 $&$ 0.18 $ \\
16 &$ 26 $&$ -2.34 $&$ 1.96 $& 56 &$ 266 $&$ 0.02 $&$ 0.61 $& 94 &$ 266 $&$ -1.49 $&$ 1.3 $& 132 &$ 98 $&$ -0.96 $&$ 1.01 $ \\
17 &$ 702 $&$ 0.98 $&$ 0.38 $& 57 &$ 555 $&$ 0.76 $&$ 0.42 $& 95 &$ 555 $&$ 0.94 $&$ 0.38 $& 133 &$ 57 $&$ -1.5 $&$ 1.32 $ \\
18 &$ 182 $&$ -0.36 $&$ 0.74 $& 58 &$ 423 $&$ 0.47 $&$ 0.49 $& 96 &$ 423 $&$ -2.22 $&$ 1.86 $& 134 &$ 106 $&$ -0.93 $&$ 0.97 $ \\
19 &$ 122 $&$ -0.78 $&$ 0.91 $& 59 &$ 3715 $&$ 2.69 $&$ 0.16 $& 97 &$ 3715 $&$ -1.88 $&$ 1.58 $& 135 &$ 182 $&$ -0.39 $&$ 0.74 $ \\
20 &$ 4637 $&$ 2.97 $&$ 0.15 $& 60 &$ 298 $&$ 0.14 $&$ 0.58 $& 98 &$ 298 $&$ 0.79 $&$ 0.41 $& 136 &$ 79 $&$ -1.19 $&$ 1.13 $ \\
21 &$ 14 $&$ -2.96 $&$ 2.67 $& 61 &$ 1832 $&$ 1.97 $&$ 0.23 $& 99 &$ 1832 $&$ -1.96 $&$ 1.62 $& 137 &$ 676 $&$ 0.96 $&$ 0.38 $ \\
22 &$ 44 $&$ -1.8 $&$ 1.51 $& 62 &$ 65 $&$ -1.41 $&$ 1.24 $& 100 &$ 65 $&$ 0.31 $&$ 0.53 $& 138 &$ 2340 $&$ 2.23 $&$ 0.21 $ \\
23 &$ 135 $&$ -0.69 $&$ 0.86 $& 63 &$ 419 $&$ 0.46 $&$ 0.49 $& 101 &$ 419 $&$ -0.19 $&$ 0.67 $& 139 &$ 3 $&$ -4.5 $&$ 5.77 $ \\
24 &$ 826 $&$ 1.15 $&$ 0.35 $& 64 &$ 68 $&$ -1.37 $&$ 1.21 $& 102 &$ 68 $&$ -0.34 $&$ 0.72 $& 140 &$ 208 $&$ -0.2 $&$ 0.69 $ \\
25 &$ 135 $&$ -0.64 $&$ 0.86 $& 65 &$ 1159 $&$ 1.48 $&$ 0.29 $& 103 &$ 1159 $&$ -1.48 $&$ 1.3 $& 141 &$ 56 $&$ -1.56 $&$ 1.34 $ \\
26 &$ 668 $&$ 0.95 $&$ 0.39 $& 66 &$ 170 $&$ -0.45 $&$ 0.77 $& 104 &$ 170 $&$ -1.04 $&$ 1.03 $& 142 &$ 241 $&$ -0.08 $&$ 0.64 $ \\
27 &$ 644 $&$ 0.88 $&$ 0.39 $& 67 &$ 815 $&$ 1.13 $&$ 0.35 $& 105 &$ 815 $&$ -1.65 $&$ 1.39 $& 143 &$ 645 $&$ 0.88 $&$ 0.39 $ \\
28 &$ 20 $&$ -2.59 $&$ 2.24 $& 68 &$ 112 $&$ -0.87 $&$ 0.94 $& 106 &$ 112 $&$ -1.3 $&$ 1.19 $& 144 &$ 540 $&$ 0.71 $&$ 0.43 $ \\
29 &$ 190 $&$ -0.34 $&$ 0.73 $& 69 &$ 707 $&$ 0.99 $&$ 0.38 $& 107 &$ 707 $&$ -1.38 $&$ 1.21 $& 145 &$ 1080 $&$ 1.43 $&$ 0.3 $ \\
30 &$ 99 $&$ -0.97 $&$ 1.01 $& 70 &$ 33 $&$ -2.09 $&$ 1.74 $& 108 &$ 33 $&$ -1.32 $&$ 1.18 $& 146 &$ 67 $&$ -1.39 $&$ 1.22 $ \\
31 &$ 60 $&$ -1.47 $&$ 1.29 $& 71 &$ 136 $&$ -0.68 $&$ 0.86 $& 109 &$ 136 $&$ 1.12 $&$ 0.35 $& 147 &$ 440 $&$ 0.51 $&$ 0.48 $ \\
33 &$ 241 $&$ -0.11 $&$ 0.64 $& 72 &$ 788 $&$ 1.12 $&$ 0.36 $& 110 &$ 788 $&$ -0.95 $&$ 0.99 $& 148 &$ 165 $&$ -0.49 $&$ 0.78 $ \\
34 &$ 996 $&$ 1.35 $&$ 0.32 $& 73 &$ 179 $&$ -0.41 $&$ 0.75 $& 111 &$ 179 $&$ -1.07 $&$ 1.07 $& 149 &$ 588 $&$ 0.8 $&$ 0.41 $ \\
35 &$ 96 $&$ -0.98 $&$ 1.02 $& 74 &$ 720 $&$ 1 $&$ 0.37 $& 112 &$ 720 $&$ -0.03 $&$ 0.62 $& 150 &$ 38 $&$ -1.95 $&$ 1.62 $ \\
36 &$ 97 $&$ -1.02 $&$ 1.02 $& 75 &$ 313 $&$ 0.15 $&$ 0.57 $& 113 &$ 313 $&$ 1.21 $&$ 0.33 $& 151 &$ 1330 $&$ 1.65 $&$ 0.27 $ \\
38 &$ 564 $&$ 0.74 $&$ 0.42 $& 76 &$ 184 $&$ -0.38 $&$ 0.74 $& 114 &$ 184 $&$ -0.04 $&$ 0.62 $& 152 &$ 120 $&$ -0.81 $&$ 0.91 $ \\
39 &$ 711 $&$ 0.98 $&$ 0.38 $& 77 &$ 358 $&$ 0.32 $&$ 0.53 $& 115 &$ 358 $&$ -0.06 $&$ 0.65 $& 153 &$ 219 $&$ -0.21 $&$ 0.68 $ \\
40 &$ 202 $&$ -0.29 $&$ 0.7 $& 78 &$ 137 $&$ -0.64 $&$ 0.85 $& 116 &$ 137 $&$ -0.94 $&$ 0.99 $& 154 &$ 298 $&$ 0.1 $&$ 0.58 $ \\
155 &$ 82 $&$ -1.17 $&$ 1.1$& 156 &$ 480 $&$ 0.6 $&$ 0.46$ \\
\hline
\end{tabular}

\end{table}
}
The estimated covariate effects, their  bias corrected estimates, their standard errors, and their $p$-values under the null of having no effects are reported in Table \ref{Table:gamma:realdata}.
From this table, we can see that the estimates and bias corrected estimates are almost the same, indicating that
the bias effect is very small in the Poisson model and it corroborates the findings of simulations.
The variables ``department" and ``seniority" are significant while ``gender" is not significant.
This indicates that the gender has no significant influence on the formation of organizational emails.
The coefficient of variable ``department" is positive,
implying that a common value increases the probability of two employees in the same department to have more email connections.
On the other hand, the coefficient of variable ``seniority" is negative, indicating that
two employees in the same seniority have less emails than those with unequal seniorities. This makes sense intuitively.

{\renewcommand{\arraystretch}{1}
\begin{table}[h]\centering
%riptsize
%\vskip5pt
\caption{The estimators of $\gamma_i$, the corresponding bias corrected estimators, the standard errors, and the $p$-values under the null $\gamma_i=0$ ($i=1, 2, 3$) for Enron email data.}
\label{Table:gamma:realdata}
\begin{tabular}{ccc ccc cc c}
%\\
\hline
Covariate %& $\gamma$
&  $\hat{\gamma}_i$ & $\hat{\gamma}_{bc, i}$ & $\hat{\sigma}_i$ &$p$-value  \\
\hline
Department          %& $\gamma_1$
&  $ 0.167 $&$ 0.167 $&$ 1.13$ &$ <0.001$\\
Gender          %& $\gamma_2$
&  $-0.006 $&$ -0.006 $&$ 1.27$&$ 0.62$\\
Seniority       %& $\gamma_3$
&  $-0.203 $&$ -0.203 $&$ 1.09$&$ <0.001$\\
\hline
\end{tabular}
\end{table}
}

\section{Proof of Lemma 4}
\label{section-lemma4}

\begin{proof}[Proof of Lemma 4]
Recall that $\pi_{ij}=\beta_i + \beta_j + z_{ij}^\top \gamma$ and
\[
F_i(\beta, \gamma) =  \sum_{j\neq i} \mu_{ij}(\beta_i + \beta_j + z_{ij}^\top \gamma )-d_i , ~~ i=1, \ldots, n.
\]
The Jacobian matrix $F^\prime_\gamma(\beta)$ of $F_\gamma(\beta)$ can be calculated as follows.
By finding the partial derivative of $F_i$ with respect to $\beta$, for $i\neq j$ we have
\[
\frac{\partial F_i(\beta, \gamma) }{ \partial \beta_j} =  \mu_{ij}^\prime (\pi_{ij}), ~~ \frac{ \partial F_i(\beta, \gamma)}{ \partial \beta_i} =  \sum_{j\neq i} \mu_{ij}^\prime (\pi_{ij}),
\]
\[
\frac{\partial^2 F_i(\beta, \gamma) }{ \partial \beta_i \partial \beta_j} = \mu_{ij}^{\prime\prime} (\pi_{ij}),~~
\frac{ \partial^2 F_i(\beta, \gamma)}{\partial \beta_i^2} =  \sum_{j\neq i} \mu_{ij}^{\prime\prime} (\pi_{ij}).
\]
When $\beta\in B(\beta^*, \epsilon_{n1})$ and $\gamma \in B(\gamma^*, \epsilon_{n2})$, by inequality \eqref{ineq-mu-keyb},  we have
\[
\left|\frac{\partial^2 F_i(\beta, \gamma) }{ \partial \beta_i \partial \beta_j} \right|\le b_{n2},~~i\neq j.
\]
Therefore,
\begin{equation}
\label{inequ:second:deri}
\left| \frac{\partial^2 F_i(\beta, \gamma) }{\partial \beta_i^2} \right|\le (n-1)b_{n2},~~
\left| \frac{\partial^2 F_i(\beta, \gamma) }{\partial \beta_j\partial \beta_i} \right| \le b_{n2}.
\end{equation}

Let
\[
\mathbf{g}_{ij}(\beta)=(\frac{\partial^2 F_i(\beta, \gamma) }{ \partial \beta_1 \partial \beta_j}, \ldots,
\frac{\partial^2 F_i(\beta, \gamma) }{ \partial \beta_n \partial \beta_j})^\top.
\]
In view of \eqref{inequ:second:deri}, we have
\[
\|\mathbf{g}_{ii}(\beta)\|_1 \le 2(n-1)b_{n2},
\]
where $\|x \|_1=\sum_i |x_i|$ for a general vector $x\in \R^n $.
Note that when $i\neq j$ and $k\neq i, j$,
\[
\frac{\partial^2 F_i(\beta, \gamma) }{ \partial \beta_k \partial \beta_j} =0.
\]
Therefore, we have
$\|{g}_{ij}(\beta)\|_1 \le 2b_{n2}$, for $j\neq i$. Consequently, for vectors ${x},{y}, {v}\subset D$, we have
\begin{eqnarray*}
& & \| [F^\prime_{\gamma}({x})]v - [F^\prime_{\gamma}({y})] v \|_\infty \\
& \le & \max_i\{\sum_j [ \frac{\partial F_i}{\partial \beta_j }({x}, \gamma) - \frac{\partial F_i}{\partial \beta_j }({y}, \gamma)] v_j \} \\
& \le & \|{v}\|_\infty \max_i \sum_{j=1}^n |\frac{\partial F_i}{\partial \beta_j }
({x}, \gamma) - \frac{\partial F_i}{\partial \beta_j }({y}, \gamma) |  \\
& = & \|{v}\|_\infty \max_i \sum_{j=1}^n |\int_0^1 [{g}_{ij}(t{x}+(1-t){y})]^\top ({x}-{y})dt | \\
& \le & 4b_{n2}(n-1)\|{v}\|_\infty\|{x}-{y}\|_\infty.
\end{eqnarray*}
It completes the proof.
\end{proof}

\section{Proof of Lemma 5}
\label{section-lemma5}

To show this lemma, we need one preliminary result.
We first introduce the concentration inequality.
We say that a real-valued random variable $X$ is {\em sub-exponential} with parameter $\kappa > 0$ if
\begin{equation*}
\E[|X|^p]^{1/p} \leq \kappa p \quad \text{ for all } p \geq 1.
\end{equation*}
Note that if $X$ is a $\kappa $-sub-exponential random variable with finite first moment, then the centered random variable $X-\E[X]$ is also sub-exponential with parameter $2 \kappa $. This follows from the triangle inequality applied to the $p$-norm, followed by Jensen's inequality for $p \geq 1$:
\begin{equation*}
\big[ \E \big( \big|X-\E[X]\big|^p\big ) ]^{1/p}
%&\leq \E[|X|^p]^{1/p} + \E\big[\big|\E[X]\big|^p\big]^{1/p} \\
\leq [ \E\big(|X|^p \big )]^{1/p} + \big|\E[X]\big|
%\leq \E[|X|^p]^{1/p} + \E[|X|]
\leq 2[\E\big( |X|^p \big )]^{1/p}.
\end{equation*}
Sub-exponential random variables satisfy the following concentration inequality.

\begin{lemma}[\cite{vershynin_2012}, Corollary~5.17]\label{Thm:ConcIneqSubExp}
Let $X_1, \dots, X_n$ be independent centered random variables, and suppose each $X_i$ is sub-exponential with parameter $\kappa_i$. Let $\kappa = \max_{1 \leq i \leq n} \kappa_i$. Then for every $\epsilon \geq 0$,
\begin{equation*}
\P\left( \left|\frac{1}{n}  \sum_{i=1}^n (X_i- \E(X_i)) \right| \geq \epsilon \right) \leq
2\exp\left[-n\gamma \cdot  \min\Big(\frac{\epsilon^2}{\kappa^2}, \: \frac{\epsilon}{\kappa} \Big) \right],
\end{equation*}
where $\gamma > 0$ is an absolute constant.
\end{lemma}

\begin{proof}[Proof of Lemma 5]
Recall that $a_{ij}-\E a_{ij}$, $1\le i<j \le n$, are independent and sub-exponential with respective parameters $h_{ij}$
and $\max_{i,j}h_{ij} \le h_n$. We set $\epsilon$ in Lemma \ref{Thm:ConcIneqSubExp} as
\begin{equation*}
\epsilon = h_n \left(\frac{2 \log (n-1)}{\gamma (n-1) } \right)^{1/2}.
\end{equation*}
Assume $n$ is sufficiently large such that $\epsilon/\kappa = \sqrt{2 \log (n-1) / \gamma (n-1)} \leq 1$.
By applying the concentration inequality in Theorem~\ref{Thm:ConcIneqSubExp}, we have
 for each $i = 1,\dots,n$,
\begin{equation*}
\begin{split}
\P\left(\frac{1}{n-1}|d_i - \E d_i | \geq h_n \left(\frac{2 \log (n-1)}{\gamma (n-1) } \right)^{1/2} \right)
&\leq 2\exp\left(- (n-1)\gamma   \cdot \frac{2 \log n}{\gamma (n-1) }\right) \\
&= \frac{2}{(n-1)^2}.
\end{split}
\end{equation*}
By the union bound,
\begin{eqnarray*}
&&\P\Bigg(\|d - \E d \|_\infty \ge h_n \sqrt{\frac{2}{\gamma}(n-1)\log(n-1)} \; \Bigg) \\
&\le & \sum_{i=1}^n \P\left(|d_i - \E d_i| \geq h_n \sqrt{\frac{2}{\gamma}(n-1)\log(n-1)} \right) \\
& \le &  \frac{2n}{(n-1)^2}.
\end{eqnarray*}
Similarly, we have
\[
\P \left( \left\| \sum_{i<j}z_{ij}(a_{ij}-\E a_{ij}) \right \|_\infty \ge h_n n\log n \right)  \le 1 - \frac{2p}{n}.
\]

\end{proof}

\section{Proof of Lemma 6}
\label{section-lemma6}

\begin{proof}[Proof of Lemma 6]
Note that $F'_{\gamma}(\beta) \in \mathcal{L}_n(b_{n0}, b_{n1})$
when $\beta\in B(\beta^*, \epsilon_{n1})$ and $\gamma \in B(\gamma^*, \epsilon_{n2})$,
 and $F_{\gamma}(\widehat{\beta}_{\gamma})$=0.
To prove this lemma, it is sufficient to show that the Kantovorich conditions for the function $F_{\gamma}(\beta)$ hold when
$D=B(\beta^*, \epsilon_{n1})$ and $\gamma\in B(\gamma^*, \epsilon_{n2})$, where $\epsilon_{n1}$ is a positive number and
$\epsilon_{n2}= o(b_{n1}^{-1}(\log n/n)^{1/2})$. %%under Conditions \ref{condition-diff-a}--\ref{condition:matrixclass-d}.
The following calculations are based on the event $E_n$:
\[
E_n = \{d: \max_i | d_i - \E d_i | = O( h_{n}(n\log n)^{1/2} ) \}.
\]
In the Newton iterative step, we set the true parameter vector $\beta^*$
as the starting point $\beta^{(0)}:=\beta^*$..

Let $V=(v_{ij})= \partial F_{\gamma}(\beta^*)/\partial \beta^\top$ and $S=\mathrm{diag}(1/v_{11}, \ldots, 1/v_{nn})$.
By Lemma 2, we have
$\aleph =\|V^{-1}\|_\infty = O( (nb_{n0})^{-1} )$.
Recall that $F_{\gamma^*}(\beta^*) = \E d -d $ and $\gamma\in B(\gamma^*, (\log n/n)^{1/2})$ and Assumption 1 holds,
Note that the dimension $p$ of $\gamma$  is a fixed constant.
If $\epsilon_{n2}=o( b_{n1}^{-1} (\log n)^{1/2} n^{-1/2})$, by the mean value theorem, we have
\begin{eqnarray*}
\| F_\gamma(\beta^*) \|_\infty & \le & \| d - \E d \|_\infty + \max_i | \sum\nolimits_{j\neq i} [\mu_{ij}(\beta^*, \gamma) - \mu_{ij}(\beta^*, \gamma^*)] |  \\
& \le & O( h_n(n\log n)^{1/2}) + \max_i \sum_{j\neq i} |\mu_{ij}^\prime( \beta^*, \bar{\gamma})| | z_{ij}^\top ( \gamma - \gamma^*)| \\
& = & O(h_n(n\log n)^{1/2}).
\end{eqnarray*}
Repeatedly utilizing  Lemma 2, we have
\begin{eqnarray*}
\delta=\| [F'_\gamma(\beta^*)]^{-1}F_\gamma(\beta^*) \|_\infty = \| [F'_\gamma(\beta^*)]^{-1}\|_\infty \|F_\gamma(\beta^*) \|_\infty
=O( \frac{h_n}{b_{n0}} \sqrt{\frac{\log n}{n}} )
\end{eqnarray*}
By Lemma 4, $F_\gamma(\beta)$ is Lipschitz continuous with Lipschitz coefficient $\lambda=4b_{n2}(n-1)$.
Therefore, if
\[
\frac{ b_{n2} h_{n} }{b_{n0}^2}=o(\sqrt{ \frac{n}{\log n} }),
\]
then
\begin{eqnarray*}
\rho =2\aleph \lambda \delta & = & O(\frac{1}{nb_{n0}})\times O( b_{n2}n )
\times O( \frac{ h_n }{b_{n0}}\sqrt{\frac{\log n}{n}} ) \\
& = & O\left( \frac{ b_{n2} h_{n} }{b_{n0}^2}\sqrt{ \frac{\log n}{n} } \right) =o(1).
\end{eqnarray*}
The above arguments verify the Kantovorich conditions.
By Lemma 3, it yields that
\[
\| \widehat{\beta}_\gamma - \beta^* \|_\infty = O\left( \frac{ h_n }{b_{n0} }\sqrt{\frac{\log n}{n}} \right).
\]
By Lemma 5, $P(E_n)\to 1$ such that the above equation holds with probability at least $1-O(n^{-1})$.
It completes the proof.
\end{proof}

\section{Proof of Lemma  7}
\label{section-lemma7}

\begin{proof}[Proof of Lemma 7]
Recall that
\[
Q_c(\gamma) = (Q_{c, 1}(\gamma), \ldots, Q_{c,p}(\gamma))^\top = \sum_{j<i} z_{ij}( \mu_{ij}(\hat{\beta}_\gamma, \gamma) - a_{ij}),
\]
and $Q_c^\prime(\gamma)$ is the Jacobian matrix of $Q_c(\gamma)$.

When causing no confusion, we write $Q_{c,k}(\gamma)$ as $Q_{c,k}$, $k=1, \ldots, p$.
Note that
\[
Q_{c,k}= \sum_{j<i} z_{ijk} (  \mu_{ij}(\hat{\beta}_\gamma, \gamma) - a_{ij} ).
\]
By finding the first order partial derivative of function $Q_{c,k}$ with respect to variable $\gamma_l$, we have
\[
\frac{\partial Q_{c,k} }{ \partial \gamma_l } =  \sum_{j<i} z_{ijk}\mu^\prime(\hat{\pi}_{ij}) \left(\frac{\partial \widehat{\beta}_{ \gamma, i}}{\partial \gamma_l } + \frac{ \partial \widehat{\beta}_{\gamma,j}}{\partial \gamma_l} + z_{ijl}\right),
\]
where $\hat{\pi}_{ij}=\widehat{\beta}_{ \gamma,i} + \widehat{\beta}_{ \gamma, j} + z_{ij}^\top \gamma$ and $\widehat{\beta}_{\gamma}=(\widehat{\beta}_{\gamma,1}, \ldots, \widehat{\beta}_{\gamma,n})^\top$.
Again, with the second order partial derivative, we have
\[
\frac{\partial^2 Q_{c,k} }{ \partial \gamma^\top \partial \gamma_l } =  \sum_{j<i} z_{ijk}
\mu^{\prime\prime}(\hat{\pi}_{ij})
\left(\frac{\partial \widehat{\beta}_{ \gamma,i}}{\partial \gamma^\top } + \frac{ \partial \widehat{\beta}_{ \gamma,j}}{\partial \gamma^\top} + z_{ij} \right)
 \left(\frac{\partial \widehat{\beta}_{\gamma,i}}{\partial \gamma_l } + \frac{ \partial \widehat{\beta}_{ \gamma,j}}{\partial \gamma_l} + z_{ijl}\right)
\]
\[
~~~~~~~ + z_{ijk}\mu^{\prime}(\hat{\pi}_{ij})
%%\frac{ e^{ \widehat{\beta}_{i, \gamma} + \widehat{\beta}_{j, \gamma} + z_{ij}^\top \gamma } }{ (1 + e^{ \widehat{\beta}_{i, \gamma} + \widehat{\beta}_{j, \gamma} + z_{ij}^\top \gamma })^2 }
 \left(\frac{\partial^2 \widehat{\beta}_{ \gamma,i}}{\partial \gamma^\top \partial \gamma_l } + \frac{ \partial^2 \widehat{\beta}_{ \gamma,j}}{\partial \gamma^\top \partial \gamma_l} \right).
\]
Recall that $\max_{i,j} \|z_{ij}\|_\infty =O(1) $ and
when $\beta \in B(\beta^*, \epsilon_{n1}), \gamma\in B(\gamma^*, \epsilon_{n2})$, we have:
\begin{equation}\label{ineq-mu-derive-b}
\max_{i,j}|\mu^\prime(\pi_{ij})|\le b_{n1}, ~~\max_{i,j}|\mu^{\prime\prime}(\pi_{ij})| \le b_{n2},~~ \max_{i,j}|\mu^{\prime\prime\prime}(\pi_{ij})| \le b_{n3}.
\end{equation}
So, we have
\begin{equation}
\label{equa:verify:c}
\| \frac{\partial^2 Q_{c,k} }{ \partial \gamma^\top \partial \gamma_l } \| =
  O \left( n^2 \left[ b_{n2} (  \| \frac{ \partial \widehat{\beta}_\gamma }{\partial \gamma^\top } \| )^2
+ b_{n1} \max_i \|  \frac{\partial^2 \widehat{\beta}_{ \gamma, i}}{\partial \gamma^\top \partial \gamma_l } \| \right]
\right).
\end{equation}
In view of \eqref{equa:verify:c}, to derive the upper bound of $\frac{\partial^2 Q_{c,k} }{ \partial \gamma^\top \partial \gamma_l }$, it is left to
bound $\frac{ \partial \widehat{\beta}_\gamma }{\partial \gamma^\top } $ and $\frac{\partial^2 \widehat{\beta}_{\gamma,i}}{\partial \gamma^\top \partial \gamma_l }$.

Recall that $F(\widehat{\beta}_\gamma, \gamma)=0$. With the derivative of function $F(\widehat{\beta}_\gamma, \gamma)$ on variable $\gamma$, we have
\begin{equation}\label{equ-verify-b}
\frac{ \partial F( \beta, \gamma)}{\partial \beta^\top}\bigg|_{\beta=\widehat{\beta}_\gamma, \gamma=\gamma} \frac{ \partial \widehat{\beta}_\gamma }{ \partial \gamma^\top } +
\frac{ \partial F( \beta, \gamma) }{ \partial \gamma^\top }\bigg|_{\beta=\widehat{\beta}_\gamma, \gamma=\gamma} =0.
\end{equation}
Thus, we have
\begin{equation}
\label{equ-one-beta-b}
\frac{ \partial \widehat{\beta}_\gamma }{ \partial \gamma^\top }
= - \left[\frac{ \partial F( \beta, \gamma)}{\partial \beta^\top}\bigg|_{\beta=\widehat{\beta}_\gamma, \gamma=\gamma} \right]^{-1}
\frac{ \partial F( \beta, \gamma) }{ \partial \gamma^\top }\bigg|_{\beta=\widehat{\beta}_\gamma, \gamma=\gamma}.
\end{equation}
To simplify notations, define
\[
V=(v_{ij})_{n\times n} : =  \frac{ \partial F( \beta, \gamma)}{\partial \beta}\bigg|_{\beta=\widehat{\beta}_\gamma, \gamma=\gamma},~~ W=(w_{ij})_{n\times n}:= V^{-1} - S,~~F:=F(\beta,\gamma),
\]
where $S=(s_{ij})_{n\times n}$ and $s_{ij}=\delta_{ij}/v_{ii}$.
Note that
\begin{equation}\label{equation-Fi-gamma}
\frac{ \partial F_i }{ \partial \gamma^\top }\bigg|_{\beta=\widehat{\beta}_\gamma, \gamma=\gamma} =  \sum_{j\neq i} z_{ij} \mu_{ij}^\prime(\widehat{\beta}_\gamma, \gamma ).
\end{equation}
By inequality \eqref{ineq-mu-derive-b}, we have
\begin{equation}
\label{eq-F-upper-bou}
\|\frac{ \partial F }{ \partial \gamma^\top }\bigg|_{\beta=\widehat{\beta}_\gamma, \gamma=\gamma} \|
\le \max_{i,k} \sum_{j\neq i} |\mu_{ij}^\prime (\widehat{\beta}_\gamma, \gamma ) | |z_{ijk} | = O( b_{n1} n).
\end{equation}
By combing \eqref{equ-one-beta-b} and \eqref{eq-F-upper-bou} and applying Lemma 2,
we have
\begin{equation}\label{eq-first-beta-upp}
\|\frac{ \partial \widehat{\beta}_\gamma }{ \partial \gamma^\top }\|_\infty
\le \|V\|_\infty \| \|\frac{ \partial F(\widehat{\beta}_\gamma, \gamma) }{ \partial \gamma^\top } \|_\infty
\le O(\frac{ 1}{nb_{n0}}) \cdot O( b_{n1}n) = O(\frac{ b_{n1}}{b_{n0}}).
\end{equation}

Next, we will evaluate $\frac{ \partial^2 \widehat{\beta}_\gamma }{ \partial \gamma_k \partial \gamma^\top  }$.
By \eqref{equ-verify-b}, we have
\[
\frac{\partial }{ \partial \gamma_k } \left[\frac{ \partial F}{\partial \beta^\top }\bigg|_{\beta=\widehat{\beta}_\gamma, \gamma=\gamma} \right]
\frac{ \partial \widehat{\beta}_\gamma }{ \partial \gamma^\top }
+ \left[\frac{ \partial F}{\partial \beta^\top}\bigg|_{\beta=\widehat{\beta}_\gamma, \gamma=\gamma} \right] \frac{ \partial^2 \widehat{\beta}_\gamma }{ \partial \gamma_k \partial \gamma^\top  }
+ \frac{\partial }{ \partial \gamma_k} \left[ \frac{\partial F }{  \partial \gamma^\top }\bigg|_{\beta=\widehat{\beta}_\gamma, \gamma=\gamma} \right] =0.
\]
It leads to that
\begin{eqnarray}
\nonumber
\frac{ \partial^2 \widehat{\beta}_\gamma }{ \partial \gamma_k \partial \gamma^\top  }
& = &  - V^{-1} \frac{\partial }{ \partial \gamma_k } \left[\frac{ \partial F}{\partial \beta^\top}\bigg|_{\beta=\widehat{\beta}_\gamma, \gamma=\gamma} \right]
\frac{ \partial \widehat{\beta}_\gamma }{ \partial \gamma^\top }
- V^{-1}\frac{\partial }{ \partial \gamma_k} \left[ \frac{\partial F }{  \partial \gamma^\top }\bigg|_{\beta=\widehat{\beta}_\gamma, \gamma=\gamma} \right]
\\
\label{equation-I1I2}
& := & - I_1 - I_2.
\end{eqnarray}
For $i\neq j$, we have
\[
\left(\frac{\partial F }{\partial \beta^\top}\bigg|_{\beta=\widehat{\beta}_\gamma, \gamma=\gamma}  \right)_{ij}=  \mu_{ij}^\prime (\widehat{\beta}_\gamma, \gamma),
\]
\[
\frac{\partial}{\partial \gamma_k}\left(\frac{\partial F }{\partial \beta^\top}\bigg|_{\beta=\widehat{\beta}_\gamma, \gamma=\gamma}  \right)_{ij}=  \mu_{ij}^{\prime\prime} (\widehat{\beta}_\gamma, \gamma)
( T_{ij}^\top \frac{\partial \widehat{\beta}_\gamma }{\partial \gamma_k} + z_{ijk}).
\]
Thus,
\begin{equation}\label{inequality-gc}
\left|\frac{\partial}{\partial \gamma_k}\left(\frac{\partial F }{\partial \beta^\top}\bigg|_{\beta=\widehat{\beta}_\gamma, \gamma=\gamma}  \right)_{ij} \right |
\le b_{n2}\left(2 \left\|\frac{ \partial \widehat{\beta}_\gamma }{ \partial \gamma^\top } \right\| +
\max_{i,j} \| z_{ij} \|_\infty \right).
\end{equation}
Note that
\[
\left(\frac{\partial F }{\partial \beta^\top}\bigg|_{\beta=\widehat{\beta}_\gamma, \gamma=\gamma}  \right)_{ii}=  \sum_{j\neq i} \mu_{ij}^\prime (\widehat{\beta}_\gamma, \gamma).
\]
By  \eqref{inequality-gc}, we have
\begin{equation}\label{inequality-gd}
\left|\frac{\partial}{\partial \gamma_k}\left(\frac{\partial F }{\partial \beta^\top}\bigg|_{\beta=\widehat{\beta}_\gamma, \gamma=\gamma}  \right)_{ii} \right |
\le (n-1)b_{n2}\left(2 \left\|\frac{ \partial \widehat{\beta}_\gamma }{ \partial \gamma^\top } \right\| + \max_{i,j} \| z_{ij} \|_\infty \right).
\end{equation}
For all $i=1, \ldots, n$ and $j=1, \ldots, p$, in view of \eqref{inequality-gc} and \eqref{inequality-gd}, we have
\begin{eqnarray*}
&&\left|\left\{\frac{\partial }{ \partial \gamma_k } \left[\frac{ \partial F}{\partial \beta^\top}\bigg|_{\beta=\widehat{\beta}_\gamma, \gamma=\gamma}\right ]
\frac{ \partial \widehat{\beta}_\gamma }{ \partial \gamma^\top } \right\}_{ij}
\right|
\\
& \le & \sum_{\ell=1}^n \left|\left\{\frac{\partial }{ \partial \gamma_k } \left[\frac{ \partial F}{\partial \beta}\bigg|_{\beta=\widehat{\beta}_\gamma, \gamma=\gamma}\right ] \right\}_{i\ell} \right|
\left| \left(\frac{ \partial \widehat{\beta}_\gamma }{ \partial \gamma^\top } \right)_{\ell j} \right|
\\
& \le & 2(n-1)b_{n2} \left\|\frac{ \partial \widehat{\beta}_\gamma }{ \partial \gamma^\top } \right\| \left(2 \left\|\frac{ \partial \widehat{\beta}_\gamma }{ \partial \gamma^\top } \right\| + \max_{i,j} \| z_{ij} \|_\infty\right)\\
& = & O( nb_{n2} \left\|\frac{ \partial \widehat{\beta}_\gamma }{ \partial \gamma^\top } \right\|^2 ).
\end{eqnarray*}
Thus,
\begin{eqnarray}
\nonumber
\| I_1 \| & = &
\left\|
V^{-1} \left\{ \frac{\partial }{ \partial \gamma_k} \left[ \frac{\partial F }{  \partial \gamma }\bigg|_{\beta=\widehat{\beta}_\gamma, \gamma=\gamma} \right]
\right\}
\left(\frac{ \partial \widehat{\beta}_\gamma }{ \partial \gamma^\top } \right)
\right\|
\\
\nonumber
& \le &  \|V^{-1}\|_\infty \times
\max_i \sum_{j=1}^p \left|\left\{\frac{\partial }{ \partial \gamma_k } \left[\frac{ \partial F}{\partial \beta^\top}\bigg|_{\beta=\widehat{\beta}_\gamma, \gamma=\gamma}\right ]
\frac{ \partial \widehat{\beta}_\gamma }{ \partial \gamma^\top } \right\}_{ij}
\right| \\
\label{equation-I1}
& = & O \left( \frac{b_{n2}}{b_{n0}}  \|\frac{\partial \widehat{\beta}_\gamma}{\partial \gamma_k}\|^2  \right).
\end{eqnarray}
Since
\[
\frac{ \partial F_i }{ \partial \gamma^\top }\bigg|_{\beta=\widehat{\beta}_\gamma, \gamma=\gamma}
=  \sum_{j\neq i} z_{ij} \mu_{ij}^\prime (\widehat{\beta}_\gamma, \gamma),
\]
we have
\[
\frac{\partial }{ \partial \gamma_k} \left[ \frac{\partial F_i }{  \partial \gamma^\top }\bigg|_{\beta=\widehat{\beta}_\gamma, \gamma=\gamma} \right]
=  \sum_{j\neq i} z_{ij} \mu_{ij}^{\prime\prime} (\widehat{\beta}_\gamma, \gamma) ( \frac{\partial \widehat{\beta}_\gamma}{\partial \gamma_k} + z_{ijk}),
\]
such that
\[
\left\|\frac{\partial }{ \partial \gamma_k} \left[ \frac{\partial F }{  \partial \gamma^\top }\bigg|_{\beta=\widehat{\beta}_\gamma, \gamma=\gamma} \right] \right\|_\infty
\le (n-1)b_{n2} (\max_{i,j} \|z_{ij}\|_\infty) ( \|\frac{\partial \widehat{\beta}_\gamma}{\partial \gamma_k}\| + \max_{i,j} \|z_{ij}\|_\infty).
\]
Consequently, we have
\begin{eqnarray}
\nonumber
\|I_2\|_\infty & = &  \| V^{-1} \frac{\partial }{ \partial \gamma_k } \left[\frac{ \partial F}{\partial \gamma^\top}\bigg|_{\beta=\widehat{\beta}_\gamma, \gamma=\gamma}\right ] \|_\infty  \\
\nonumber
& \le &  \|V^{-1}\|_\infty \|\frac{\partial }{ \partial \gamma_k } \left[\frac{ \partial F}{\partial \gamma^\top}\bigg|_{\beta=\widehat{\beta}_\gamma, \gamma=\gamma}\right ] \|_\infty
\\
& = & O( \frac{1}{nb_{n0}} + \frac{b_{n1}^2}{nb_{n0}^3} ) \times (n-1)b_{n2} \kappa_n ( \|\frac{\partial \widehat{\beta}_\gamma}{\partial \gamma_k}\| + \kappa_n)\\
\label{equation-I2}
& = & O\left( \frac{b_{n2}}{b_{n0}} \times   \|\frac{\partial \widehat{\beta}_\gamma}{\partial \gamma_k}\|_\infty   \right).
\end{eqnarray}
By combining \eqref{equation-I1I2}, \eqref{equation-I1} and \eqref{equation-I2}, it yields that
\begin{equation}
\label{eq-second-upper}
\| \frac{ \partial^2 \widehat{\beta}_\gamma }{ \partial \gamma_k \partial \gamma^\top  } \| = O \left( \frac{b_{n2}}{b_{n0}}  \|\frac{\partial \widehat{\beta}_\gamma}{\partial \gamma_k}\|^2  \right).
\end{equation}
Consequently, in view of  \eqref{equa:verify:c}, \eqref{eq-first-beta-upp} and \eqref{eq-second-upper}, we have
\begin{eqnarray}
\nonumber
\| \frac{\partial^2 Q_{c,k} }{ \partial \gamma^\top \partial \gamma_l } \|
 & = & O ( n^2b_{n2}  \|\frac{\partial \widehat{\beta}_\gamma}{\partial \gamma_k}\|^2
 + n^2 b_{n1} \cdot \frac{b_{n2}}{b_{n0}}\|\frac{\partial \widehat{\beta}_\gamma}{\partial \gamma_k}\|^2 \\
\label{eq-pro2-final}
& = &  O\left(\frac{n^2 b_{n1}^{3}b_{n2}}{b_{n0}^3}\right).
\end{eqnarray}

Note that
\begin{equation}\label{inequality-ga}
|\sum_{j=1}^p \{ [Q_c^\prime(x)]_{ij} - [Q_c^\prime(y)]_{ij} \} v_j | \le \| v \|_1 \max_{i,j} | [Q_c^\prime(x)]_{ij} - [Q_c^\prime(y)]_{ij} |.
\end{equation}
By the mean value theorem, we have
\begin{eqnarray}
\nonumber
|[Q_{c}^\prime(x)]_{k\ell} - [Q_{c}^\prime(y)]_{k\ell}|   & = & \left|\frac{ \partial[ Q_{c}^\prime (\gamma)]_{k\ell} }{ \partial \gamma^\top } \bigg|_{\gamma=t} (x-y) \right|
\\
\label{inequality-gb}
& \le &   \left\| \frac{\partial^2 Q_{c,k}(\gamma) }{ \partial \gamma^\top  \partial \gamma_l } \bigg|_{\gamma=t} \right\|_1 \|x-y\|_\infty,
\end{eqnarray}
where $t= \alpha x + (1-\alpha) y$ for some $\alpha\in(0,1)$.
By combining inequalities \eqref{inequality-ga}, \eqref{inequality-gb} and \eqref{eq-pro2-final}, we have
\begin{equation}
\label{check:Qc}
\| [Q_c^\prime (x)]v  - [Q_c^\prime (y)] v \|_\infty \le \lambda \| x - y \|_\infty \| v\|_\infty,
\end{equation}
where
\[
\lambda=O(n^2 b_{n1}^{3} b_{n2} b_{n0}^{-3} ).
\]
This completes the proof.
\end{proof}

\section{Proof of Lemma 8}
\label{section-lemma8}

\begin{proof}[Proof of Lemma 8]
Recall that $F_i(\beta^*, \gamma^*) = \sum_{j\neq i} ( \mu_{ij}(\beta^*, \gamma^*) -
a_{ij})$, $i=1, \ldots, n$.
By applying a second order Taylor expansion to $H(\widehat{\beta}_{\gamma^*}, \gamma^*)$, we have
\begin{equation}
\label{equ-lemma-gamma-b}
F(\widehat{\beta}_{\gamma^*}, \gamma^*)  = F(\beta^*, \gamma^*) +  \frac{\partial F(\beta^*, \gamma^*)}{\partial \beta^\top } (\widehat{\beta}^* - \beta^*)
+ \frac{1}{2} \left[\sum_{k=1}^{n-1} (\widehat{\beta}^*_k - \beta^*_k) \frac{\partial^2F(\bar{\beta}^*, \gamma^*)}{\partial \beta_k \partial \beta^\top} \right]\times (\widehat{\beta}^* - \beta^*),
\end{equation}
where $\bar{\beta}^*$ lies between $\widehat{\beta}^*$ and $\beta^*$. We evaluate the last term in the above equation row by row.
Its $\ell$th row is
\begin{equation}\label{definition-R}
R_\ell := \frac{1}{2} (\widehat{\beta}^* - \beta^*)^\top  \frac{\partial^2 F_\ell(\bar{\beta}^*, \gamma^*)}{\partial \beta \partial \beta^\top} (\widehat{\beta}^* - \beta^*),~~\ell=1, \ldots, n.
\end{equation}
A directed calculation gives that
\[
\frac{\partial^2 F_\ell(\bar{\beta}^*, \gamma^*)}{\partial \beta_i \partial \beta_j} =
\begin{cases}
 \sum_{t\neq i} \mu^{\prime\prime}(\bar{\pi}_{it}),  &   \ell =i=j  \\
\mu^{\prime\prime}(\bar{\pi}_{\ell j}), & \ell=i, i\neq j; \ell=j, i\neq j \\
0, & \ell \neq i \neq j.
\end{cases}
\]
By \eqref{ineq-mu-keyb},  we have
\begin{eqnarray*}
\max_{\ell=1, \ldots, n} 2|R_\ell| &  \le & \max_{\ell=1, \ldots, n} \sum_{1\le i<j \le n-1} |\frac{\partial^2 F_\ell(\bar{\beta}^*, \gamma^*)}{\partial \beta_i \partial \beta_j} |
\|\widehat{\beta}^* - \beta^*\|^2 \\
& \le & 2b_{n2}(n-1)\|\widehat{\beta}^* - \beta^*\|^2.
\end{eqnarray*}
By Lemma 6, we have that
\begin{equation}
\label{eq:rk}
\max_{\ell=1, \ldots, n} |R_\ell|
= O_p\left( \frac{b_{n2}h_n^2 \log n}{b_{n0}^2}  \right).
\end{equation}
Let $R=(R_1, \ldots, R_{n})^\top$ and $V=-\partial F(\beta^*, \gamma^*)/\partial \beta^\top$.
Since $H(\widehat{\beta}^*,\gamma^*)=0$,
by \eqref{equ-lemma-gamma-b}, we have
\begin{equation}\label{eq-expression-beta-star}
\widehat{\beta}^* - \beta^* =  V^{-1} F(\beta^*, \gamma^*) + V^{-1} R .
\end{equation}
Note that $V\in \mathcal{L}_n(b_{n0}, b_{n1})$.
By \eqref{eq:rk} and Lemma 1, we have
\begin{eqnarray*}
\|V^{-1}R \|_\infty & \le & \|V^{-1} \|_\infty\|R \|_\infty
 =  O_p( \frac{b_{n2} h_n^2 \log n}{nb_{n0}^3}  ).
\end{eqnarray*}
\end{proof}

\section{Proof of Lemma 9}
\label{section-lemma9}

\begin{proof}[Proof of Lemma 9]
For convenience, write
\[
\widehat{\beta}^* = \widehat{\beta}_{\gamma} , ~~V(\beta, \gamma)=-\frac{\partial F(\beta, \gamma)}{ \partial \beta^\top },~~
Q^\prime_{\beta,\ell}:= \frac{\partial Q_\ell (\beta, \gamma)}{ \partial \beta^\top },~\ell=1, \ldots, p.
\]
When evaluating functions $f(\beta, \gamma)$ on $(\beta, \gamma)$ at its true value $(\beta^*, \gamma^*)$,
we suppress the argument $(\beta^*, \gamma^*)$. This is, write $Q^\prime_{\beta,\ell}=Q^\prime_{\beta,\ell}(\beta^*, \gamma^*) $, etc.
Note that $V\in \mathcal{L}_n(b_{n0}, b_{n1})$. Let $W=V^{-1}-S$.
By Lemma 8, we have
\begin{equation}\label{eq-Q-identity}
- \frac{\partial Q_\ell (\beta^*, \gamma^*)}{ \partial \beta^\top } ( \widehat{\beta}^* - \beta^* )
= Q^\prime_{\beta,\ell} ( V^{-1} F + V^{-1} R )
= Q^\prime_{\beta,\ell} ( S F + WF +  V^{-1} R ).
\end{equation}
We will bound $Q^\prime_{\beta,\ell} S F$, $Q^\prime_{\beta,\ell}WF$ and $Q^\prime_{\beta,\ell}V^{-1} R$
in turn as follows.
Let $z_* = \max_{i,j} \|z_{ij}\|_\infty$.
A direct calculation gives
\[
Q^\prime_{\beta,\ell,i} =  \sum_{j=1, j\neq i}^n z_{ij \ell } \mu^\prime_{ij} ( \pi_{ij}^* ),
\]
such that
\begin{equation}\label{eq-Q-upper-bound}
|Q^\prime_{\beta,\ell,i}| \le (n-1) z_* b_{n1}.
\end{equation}
Thus, by Lemmas 2 and 8, we have
\begin{eqnarray}
\nonumber
|Q^\prime_{\beta,\ell} V^{-1} R | & \le & \sum_i |Q^\prime_{\beta,\ell,i}| \|V^{-1} R\|_\infty  \\
\label{eq-QVR-upper-bound}
& \le & n(n-1)  b_{n1} O_p( \frac{ b_{n2} h_n^2 \log n }{ nb_{n0}^3 } )= O_p( \frac{ nb_{n2}b_{n1} h_n^2 \log n }{ b_{n0}^3 }).
\end{eqnarray}
Then we bound $Q^\prime_{\beta,\ell} SF$. A direct calculation gives that
\begin{equation}\label{eq-QSH}
Q^\prime_{\beta,\ell} SF = \sum_{i=1}^{n} \frac{ Q^\prime_{\beta,\ell, i}}{v_{ii}}F_i=\sum_{i=1}^{n} c_iH_i,
\end{equation}
where
\[
c_i = \frac{ Q^\prime_{\beta,\ell, i}}{v_{ii}}, i=1, \ldots, n.
\]
It is easy to show that
\[
\max_{i=1, \ldots, n} |c_i| \le  \frac{  z_*  b_{n1}}{ b_{n0} }.
\]
By expressing  $Q^\prime_{\beta,\ell} SF$ as a sum of $a_{ij}$s, we have
\[
Q^\prime_{\beta,\ell} SF = 2 \sum_{1\le i < j \le n} c_i (\mu_{ij} - a_{ij}),
\]
Note $a_{ij}$ ($i<j$) is independent and bounded by $h_n z_*$.
By applying the concentration inequality for subexponential random variables to the above sum, we have
\begin{equation}
\label{eq-QSH-upper-bound}
| Q^\prime_{\beta,\ell} SF | = O_p( h_n n\log n ).
\end{equation}

Finally, we bound $Q^\prime_{\beta,\ell} WF$.
Let
\[
\sigma_n^2 = \max_{i,j} n^2 |(W^\top \mathrm{Cov} (F) W)_{ij}|.
\]
Therefore, by \eqref{eq-Q-upper-bound}, we have
\begin{eqnarray*}
\mathrm{ Var} ( Q^\prime_{\beta, \ell} W F ) & = & [ Q^\prime_{\beta, \ell}]^\top W^\top \mathrm{Cov}(F) W  Q^\prime_{\beta, \ell} \\
& = &  \sum_{i,j} Q^\prime_{\beta, \ell, i} (W^\top \mathrm{Cov}(F) W)_{ij} Q^\prime_{\beta, \ell,j}  \\
& = & O\left( n^2\times n^{-2} \sigma_n^2 \times  b_{n1}^2 n^2  \right) \\
& = & O( n^2  b_{n1}^2 \sigma_n^2 ).
\end{eqnarray*}
By  Chebyshev's inequality, we have
\[
\P (  |Q^\prime_{\beta, \ell} W H| > n  b_{n1}^2 \sigma_n (\log n)^{1/2} ) \le \frac{O( n^2 \sigma_n^2 b_{n1}^2 b_{n0}^{-3} ) }
{n^2  b_{n1}^2 b_{n0}^{-3} \sigma_n^2 \log n } \to 0.
\]
It leads to
\begin{equation}
\label{eq-QWH-upper-bound}
Q^\prime_{\beta, \ell} W H = O_p(n  b_{n1} \sigma_n(\log n)^{1/2}).
\end{equation}
By combining \eqref{eq-Q-identity}, \eqref{eq-QVR-upper-bound} \eqref{eq-QSH-upper-bound} , \eqref{eq-QWH-upper-bound},
it yields
\begin{eqnarray*}
&&\max_{\ell=1, \ldots, p} | Q^\prime_{\beta, \ell} (\widehat{\beta}^* - \beta^*) |\\
&=& O_p( \frac{ nb_{n2}b_{n1} h_n^2 \log n }{ b_{n0}^3 })
+ O_p( h_n  n\log n )
+ O_p(n  b_{n1} \sigma_n (\log n)^{1/2} ) \\
& = &  O_p\left(  nb_{n1} \log n (\frac{h_n^2b_{n2}}{b_{n0}^{3}}+ \sigma_n) \right).
\end{eqnarray*}

In the case of $V=\mathrm{Cov} (F)$, the equation \eqref{eq-QWH-upper-bound} could be simplified. Denote $W=V^{-1} -S$. Then we have
\[
\mathrm{Cov}(WF) = W^\top \mathrm{Cov} (F) W = (V^{-1} -S ) V (V^{-1} -S ) = V^{-1} -S  + SVS - S.
\]
A direct calculation gives that
\[
(SVS-S)_{ij} =  \frac{ ( 1-\delta_{ij})v_{ij}}{v_{ii}v_{jj}}.
\]
By Lemma 1, we have
\[
|(W^\top \mathrm{Cov} (F) W)_{ij}| = O\left( \frac{ b_{n1}^2 }{n^2 b_{n0}^3 } \right).
\]
Then, we have
\begin{equation*}
Q^\prime_{\beta, \ell} W F = O_p\left( \frac{ n b_{n1}^2 }{ b_{n0}^{3/2}} \right),
\end{equation*}
which leads to the simplification:
\[
\max_{\ell=1, \ldots, p} | Q^\prime_{\beta, \ell} (\widehat{\beta}^* - \beta^*) |
 =   O_p\left(   \frac{h_n^2b_{n1} b_{n2}n \log n}{b_{n0}^{3}} \right).
\]
It completes the proof.
\end{proof}

\section{Proofs for Theorem 2}
\label{section-theorem2}

\begin{proof}[Proof of Theorem 2]

To simplify notations, write  $\mu_{ij}^\prime = \mu^\prime(\beta_i^* + \beta_j^* + z_{ij}^\top \gamma^*)$ and
\[
V= \frac{ \partial F(\beta^*, \gamma^*)}{\partial \beta^\top}, ~~ V_{\gamma\beta} = \frac{ \partial F(\beta^*, \gamma^*)}{\partial \gamma^\top}.
\]
Let $\pi_{ij}^* = \beta_i^* + \beta_j^* + z_{ij}^\top \gamma^*$ and $\widehat{\pi}_{ij} = \widehat{\beta}_i + \widehat{\beta}_j + z_{ij}^\top \widehat{\gamma}$.
By a second order Taylor expansion, we have
\begin{equation}
\label{equ-Taylor-exp}
\mu( \widehat{\pi}_{ij} ) - \mu( \pi_{ij}^* )
= \mu_{ij}^\prime (\widehat{\beta}_i-\beta_i)+\mu_{ij}^\prime (\widehat{\beta}_j-\beta_j) + \mu_{ij}^\prime z_{ij}^\top ( \widehat{\gamma} - \gamma)
+ g_{ij},
\end{equation}
where
\[
g_{ij}= \frac{1}{2} \begin{pmatrix}
\widehat{\beta}_i-\beta_i^* \\
\widehat{\beta}_j-\beta_j^* \\
\widehat{\gamma} - \gamma^*
\end{pmatrix}^\top%%(\widehat{\beta}_i-\beta_i^*, \widehat{\beta}_j-\beta_j^*,  ( \widehat{\gamma} - \gamma^* )^\top  )
\begin{pmatrix}
\mu^{\prime\prime}_{ij}( \tilde{\pi}_{ij} ) & -\mu^{\prime\prime}_{ij}( \tilde{\pi}_{ij} )
& \mu^{\prime\prime}_{ij}( \tilde{\pi}_{ij} ) z_{ij}^\top \\
-\mu^{\prime\prime}_{ij}( \tilde{\pi}_{ij} ) & \mu^{\prime\prime}_{ij}( \tilde{\pi}_{ij} )
& -\mu^{\prime\prime}_{ij}( \tilde{\pi}_{ij} ) z_{ij}^\top \\
\mu^{\prime\prime}_{ij}( \tilde{\pi}_{ij} ) z_{ij}^\top
& -\mu^{\prime\prime}_{ij}( \tilde{\pi}_{ij} ) z_{ij}^\top & \mu^{\prime\prime}_{ij}( \tilde{\pi}_{ij} ) z_{ij}z_{ij}^\top
\end{pmatrix}
\begin{pmatrix}
\widehat{\beta}_i-\beta_i^* \\
\widehat{\beta}_j-\beta_j^* \\
\widehat{\gamma} - \gamma^*
\end{pmatrix},
\]
and $\tilde{\pi}_{ij}$ lies between $\pi_{ij}^*$ and $\widehat{\pi}_{ij}$.
By calculations, $g_{ij}$ can be simplified as
\begin{eqnarray*}
g_{ij} & = &  \mu^{\prime\prime}( \tilde{\pi}_{ij} ) [(\widehat{\beta}_i-\beta_i)^2 +  (\widehat{\beta}_j-\beta_j)^2 + 2(\widehat{\beta}_i-\beta_i)(\widehat{\beta}_j-\beta_j)]
\\
&& + 2\mu^{\prime\prime}( \tilde{\pi}_{ij} ) z_{ij}^\top ( \widehat{\gamma} - \gamma) (\widehat{\beta}_i-\beta_i+\widehat{\beta}_j-\beta_j) +
( \widehat{\gamma} - \gamma)^\top  \mu^{\prime\prime}( \tilde{\pi}_{ij} ) z_{ij}z_{ij}^\top( \widehat{\gamma} - \gamma)
\end{eqnarray*}
Recall that $z_*:= \max_{i,j} \| z_{ij} \|_\infty =O(1)$.
Note that $|\mu^{\prime\prime}(\pi_{ij})|\le b_{n2}$ when $\beta \in B(\beta^*, \epsilon_{n1})$ and $\gamma \in B(\gamma^*, \epsilon_{n2})$.
So we have
\begin{equation*}
\begin{array}{rcl}
|g_{ij}| & \le & 4b_{n2} \| \widehat{\beta} - \beta^*\|_\infty^2 + 2b_{n2}\| \widehat{\beta} - \beta^*\|_\infty \| \widehat{\gamma}-\gamma^* \|_1 \kappa_n + b_{n2} \| \| \widehat{\gamma}-\gamma^* \|_1^2 \kappa_n^2 \\
& \le & 2b_{n2}[4 \| \widehat{\beta} - \beta^*\|_\infty^2+  \| \widehat{\gamma}-\gamma^* \|_1^2 z_*^2].
\end{array}
\end{equation*}
Let  $g_i=\sum_{j\neq i}g_{ij}$, $g=(g_1, \ldots, g_n)^\top$.
If (4.11) in the main text holds %%\eqref{eq-theorema-ca}
and
\[
\eta_n:=\frac{  b_{n1}^2 \kappa_n^2 \log n}{n} (\frac{h_n^2 b_{n2}}{b_{n0}^3} + \sigma_n)^2 = o(1),
\]
then
\begin{equation}
\label{inequality-gij}
\max_{i=1, \ldots, n} |g_i| \le n\max_{i,j} |g_{i,j}| = O( \frac{ h_n^2 b_{n2} \log n}{ b_{n0}^2 } )
+  O_p(b_{n2} \eta_n \log n ) = O_p( \frac{ h_n^2 b_{n2} \log n}{ b_{n0}^2 }).
\end{equation}
Writing \eqref{equ-Taylor-exp} into a matrix form, it yields
\[
d - \E d = V(\widehat{\beta} - \beta^*) + V_{\gamma\beta} (\widehat{\gamma}-\gamma^*) + g,
\]
which is equivalent to
\begin{equation}
\label{expression-beta}
\widehat{\beta} - \beta^* = V^{-1}(d - \E d) + V^{-1}V_{\gamma\beta} (\widehat{\gamma}-\gamma^*) + V^{-1} g.
\end{equation}
We bound the last two remainder terms in the above equation as follows.
Let $W=V^{-1} - S$.
Note that $(S g)_i =  g_i/v_{ii}$ and $(n-1)b_{n0} \le v_{ii} \le (n-1)b_{n1}$.
By Lemma 2 in the main text, % \ref{lemma-tight-V},
we have
\begin{eqnarray}
\label{eq-Vg-upper}
\| V^{-1} g \|_\infty \le \|V^{-1}\|_\infty \|g \|_\infty
= O( \frac{1}{nb_{n0}} \times \frac{ h_n^2 b_{n2} \log n}{ b_{n0}^2 } ).
\end{eqnarray}

Note that the $i$th of $V_{\gamma\beta}$ is $\sum_{j=1,j\neq 1}^n \mu^\prime_{ij} z_{ij}^\top$.
So we have
\[
\| V_{\gamma\beta}(\widehat{\gamma}-\gamma^*) \|_\infty \le (n-1) z_*\|\widehat{\gamma}-\gamma^*\|_1
=  O_p( \kappa_n  b_{n1} \log n (\frac{b_{n2}h_n^2}{b_{n0}^{3}}+ \sigma_n)).
\]
By Lemma 2 in the main text, %\ref{pro:inverse:appro},
we have
\begin{equation}
\label{equ-theorem3-3}
\begin{array}{rcl}
\|V^{-1}V_{\gamma\beta} (\widehat{\gamma}-\gamma^*)\|_\infty    \le
\|V^{-1}\|_\infty \|V_{\gamma\beta} (\widehat{\gamma}-\gamma^*)\|_\infty
= O_p\left( \frac{\kappa_n  b_{n1} \log n}{nb_{n0}} (\frac{b_{n2}h_n^2}{b_{n0}^{3}}+ \sigma_n) \right).
\end{array}
\end{equation}
Since
$\max_i |(W\Omega W^\top)_{ii}| \le \sigma_n^2/n^2$,
we have
\begin{equation}
\label{equ-theorem3-dd}
\P( [W(d - \E d)]_i > \sigma_n \log n/n ) \le \frac{n^2}{\sigma_n^2(\log n)^2} |\mathrm{Var}\{[W(d - \E d)]_i\}| =\frac{1}{(\log n)^2}.
\end{equation}
Consequently, by combining \eqref{expression-beta}, \eqref{eq-Vg-upper}, \eqref{equ-theorem3-3} and \eqref{equ-theorem3-dd},
we have
\[
\widehat{\beta}_i - \beta^*_i = [S(d - \E d)]_i +  O_p( \frac{\kappa_n b_{n1}\log n}{ nb_{n0}} (\frac{b_{n2}h_n^2}{b_{n0}^{3}}+ \sigma_n)).
\]
It completes the proof.

\end{proof}

\section{Proof of Theorem 3}  %%~\ref{theorem-central-b}
\label{section-theorem3}

\begin{proof}[Proof of Theorem 3]  %%\ref{theorem-central-b}
Assume that the conditions in Theorem 1  %%ref{Theorem:con}
hold.
A mean value expansion gives
\[
 Q_c( \widehat{\gamma} ) - Q_c(\gamma^*) =  \frac{\partial Q_c(\bar{\gamma}) }{ \partial \gamma^\top }  (\widehat{\gamma}-\gamma^*),
\]
where $\bar{\gamma}$ lies between $\gamma^*$ and $\widehat{\gamma}$.
By noting that $Q_c( \widehat{\gamma} )=0$, we have
\[
\sqrt{N}(\widehat{\gamma} - \gamma^*) =
\Big[ \frac{1}{N}  \frac{\partial Q_c(\bar{\gamma}) }{ \partial \gamma^\top } \Big]^{-1}
\times \frac{Q_c(\gamma^*}{\sqrt{N}}). %%\times \Big[\frac{1}{\sqrt{N}}  \sum_{j< i} s_{\gamma_{ij}}(\widehat{\beta}(\gamma^*), \gamma^*) \Big].
\]
Note that the dimension of $\gamma$ is fixed. By Theorem 1 and (4.10) in the main text, %\ref{Theorem:con} and \eqref{equation-H-appro},
we have
\[
\frac{1}{N}  \frac{\partial Q_c(\bar{\gamma}) }{ \partial \gamma^\top }
\stackrel{p}{\to } \bar{H}:=\lim_{N\to\infty} \frac{1}{N}H(\beta^*, \gamma^*).
\]
Write $\widehat{\beta}^*$ as $\widehat{\beta}_{\gamma^*}$ for convenience. Therefore,
\begin{equation}\label{eq:theorem4:aa}
\sqrt{N} (\widehat{\gamma} - \gamma^*) =  \bar{H}^{-1} \times (-\frac{Q( \widehat{\beta}^*, \gamma^*)}{\sqrt{N}})   + o_p(1).
%% \Big[ \frac{1}{\sqrt{N}}  \sum_{j< i}s_{\gamma_{ij}} ( \widehat{\beta}^*, \gamma^*   )\Big] + o_p(1).
\end{equation}
By applying a third order Taylor expansion to $Q( \widehat{\beta}^*, \gamma^*)$, it yields
\begin{equation}\label{eq:gamma:asym:key}
-\frac{1}{\sqrt{N}}  Q(\widehat{\beta}^*, \gamma^*) = S_1 + S_2 + S_3,
\end{equation}
where
\begin{equation*}
\begin{array}{l}
S_1  = - \frac{1}{\sqrt{N}}  Q(\beta^*, \gamma^* )
- \frac{1}{\sqrt{N}}
\Big[\frac{\partial  Q(\beta^*, \gamma^* ) }{\partial \beta^\top } \Big]( \widehat{\beta}^* - \beta^* ), \\
S_2  =  - \frac{1}{2\sqrt{N}} \sum_{k=1}^{n} \Big[( \widehat{\beta}_k^* - \beta_k^* )
\frac{\partial^2 Q(\beta^*, \gamma^* ) }{ \partial \beta_k \partial \beta^\top }
\times ( \widehat{\beta}^* - \beta^* ) \Big],  \\
S_3  = - \frac{1}{6\sqrt{N}} \sum_{k=1}^{n} \sum_{l=1}^{n} \{ (\widehat{\beta}_k^* - \beta_k^*)(\widehat{\beta}_l^* - \beta_l^*)
\Big[   \frac{ \partial^3 Q(\bar{\beta}^*, \gamma^*)}{ \partial \beta_k \partial \beta_l \partial \beta^\top } \Big]
(\widehat{\beta}^*  - \beta^* )\},
\end{array}
\end{equation*}
and $\bar{\beta}^*=t\beta^*+(1-t)\widehat{\beta}^*$ for some $t\in(0,1)$.
Similar to the proof of Theorem 4 in \cite{Graham2017}, we will show that
(1) $S_2$ is the bias term having a non-zero probability limit; (2) $S_3$ is an asymptotically negligible
remainder term.

We first evaluate the term $S_3$.
We calculate $g_{ijk}=\frac{ \partial^3 Q (\beta, \gamma ) }{ \partial \beta_k \partial \beta_i \partial \beta_k }$ according to the indices $i,j,k$ as follows.
Observe that $g_{ijk}=0$ when $i,j,k$ are different numbers because $\mu_{ij}$ only has two arguments $\beta_i$ and $\beta_j$ and its third partial derivative on three different $\beta_i$, $\beta_j$ and $\beta_k$ is zero .
So there are only two cases below in which  $g_{ijk}\neq 0$. \\
(1) Only two values among three indices $i, j, k$ are equal.
If $k=i; i\neq j$,
$g_{ijk}=z_{ij}\frac{\partial^3 \mu_{ij}}{\partial \pi_{ij}^3}$; for other cases, the results are similar.\\
(2) Three values are equal.
$g_{kkk}=\sum_{i\neq k} z_{ki}\frac{\partial^3 \mu_{ki}}{\partial \pi_{ki}^3}$ . \\
Therefore, we have
\begin{eqnarray*}
S_3&=&\frac{1}{6\sqrt{N}}   \sum_{k, l, h}
\frac{ \partial^3 Q(\bar{\beta}^*, \gamma^*)}{ \partial \beta_k \partial \beta_l \partial \beta_h }(\widehat{\beta}_k^* - \beta_k^*)(\widehat{\beta}_l^* - \beta_l^*)
(\widehat{\beta}_h^* - \beta_h^*) \\
& = & \frac{1}{6\sqrt{N}}   \left \{\sum_{i<j}
 \frac{ \partial^3 Q(\bar{\beta}^*, \gamma^*)}{ \partial \beta_i^2 \partial \beta_j}(\widehat{\beta}_i^* - \beta_i^*)^2(\widehat{\beta}_j^* - \beta_j^*)
 +  \frac{ \partial^3 Q(\bar{\beta}^*, \gamma^*)}{ \partial \beta_j^2 \partial \beta_i}(\widehat{\beta}_j^* - \beta_j^*)^2(\widehat{\beta}_i^* - \beta_i^*)
 \right.
 \\
 &&
 \left.
 +\sum_i \frac{ \partial^3 Q(\bar{\beta}^*, \gamma^*)}{ \partial \beta_i^3 }(\widehat{\beta}_i^* - \beta_i^*)^3
 \right \}.
\end{eqnarray*}
So
\begin{eqnarray*}
\|S_3\|_\infty & \le &\frac{4}{3\sqrt{N}} \times \max_{i,j} \left\{|\frac{\partial^3 \mu_{ij}(\bar{\beta}^*, \gamma^*))}{\partial \pi_{ij}^3}| \| z_{ij} \|_\infty \right\}
\times \frac{n(n-1)}{2} \| \widehat{\beta}^* - \beta\|_\infty^3.
\end{eqnarray*}
By Lemma 6, %%\ref{lemma-a},
we have
\[
\|S_3\|_\infty =  O_p( \frac{b_{n3}h_{n}^3(\log n)^{3/2}}{n^{1/2}b_{n0}^3} ).
\]

Similar to the calculation in the derivation of the asymptotic bias in Theorem 4 in \cite{Graham2017}, we have
$S_2=B_*+o_p(1)$, where $B_*$ is defined at (4.12) in the main text. %%\eqref{defintion-Bias}.
\iffalse
\begin{equation}\label{definition:Bstar}
B_*=\lim_{n\to\infty} \frac{1}{2\sqrt{N}} \sum_{i=1}^n \frac{  \sum_{j\neq i} z_{ij} \mu^{\prime\prime}_{ij}(\pi_{ij}^*)   }
{  \sum_{j\neq i} \mu^{\prime}_{ij}(\pi_{ij}^*)},
\end{equation}
where $\pi_{ij}^*=\beta_i^* + \beta_j^* + z_{ij}^\top \gamma^*$.
\fi

Recall that $V= \partial F(\beta^*, \gamma^*)/\partial \beta^\top$ and
$V_{Q\beta}:=\partial Q(\beta^*, \gamma^*)/ \partial \beta^\top $.
By noting that
\[
d - \E d =  \sum_{1\le i<j<n} (a_{ij} - \E a_{ij}) T_{ij},
\]
we have
\[
-[Q(\beta^*, \gamma^*) - V_{Q\beta} V^{-1}( d - \E d)] = \sum_{1\le i<j \le n} (a_{ij}-\E a_{ij}) ( z_{ij} - V_{Q\beta} V^{-1} T_{ij} ).
\]
Similar to the calculation in the derivation of the asymptotic expression of $S_1$ in \cite{Graham2017}, we have
\[
S_1 = \frac{1}{\sqrt{N}} \sum_{j< i} s_{ij}(\beta^*, \gamma^*) + o_p(1),
\]
Therefore, it shows that equation \eqref{eq:gamma:asym:key} is equal to
\begin{equation}\label{eq:proof:4-a}
\frac{1}{\sqrt{N}} \sum_{j< i} s_{ij}( \widehat{\beta}^*, \gamma^* )
= \frac{1}{\sqrt{N}} \sum_{j< i} s_{ij}( \beta^*, \gamma^* ) + B_* + o_p(1),
\end{equation}
with $\frac{1}{\sqrt{N}} \sum_{i=1}^n \sum_{j\neq i} s_{\gamma_{ij}}^*( \beta^*, \gamma^* )$ equivalent to the first two terms in \eqref{eq:gamma:asym:key}
and $B_*$ the probability limit of the third term in \eqref{eq:gamma:asym:key}.

Substituting \eqref{eq:proof:4-a} into \eqref{eq:theorem4:aa} then gives
\[
\sqrt{N}(\widehat{\gamma}- \gamma^*) =  \bar{H}^{-1} B_* + \bar{H}^{-1} \times \frac{1}{\sqrt{N}}  \sum_{j< i}
s_{ij} (\beta^*, \gamma^*) + o_p(1).
\]
It completes the proof.
\end{proof}

\section{Proof of (15)} %% $\frac{1}{n^2} H(\beta, \gamma^*) =\frac{1}{n^2} H(\beta^*, \gamma^*) + o(1)$
\label{section-proof413}

Recall that $\pi_{ij}= z_{ij}^\top \gamma + \beta_i + \beta_j$, $\mu_{ij}(\pi_{ij})=\E a_{ij}$  and
$T_{ij}$ is an $n$-dimensional vector with $i$th and $j$th elements $1$ and other elements $0$.
By calculations, we have
\begin{eqnarray*}
\frac{ \partial Q(\beta, \gamma) }{ \partial \gamma^\top } =  \sum_{j<i} z_{ij} z_{ij}^\top  \mu_{ij}^{\prime}( \pi_{ij} ),
\\
\frac{ \partial Q(\beta, \gamma)}{ \partial \beta^\top} =  \sum_{j<i}   z_{ij} T_{ij}^\top \mu_{ij}^{\prime}( \pi_{ij}),
\\
\frac{\partial F(\beta, \gamma)}{ \partial \gamma^\top } =
\begin{pmatrix} \sum_{j\neq 1} z_{1j}^\top  \mu_{1j}^{\prime}( \pi_{1j} ) \\
\vdots \\
\sum_{j\neq n} z_{nj}^\top  \mu_{nj}^{\prime}( \pi_{nj} )
\end{pmatrix}.
\end{eqnarray*}
Note that
\[
H(\beta, \gamma^*) = \frac{ \partial Q(\beta, \gamma^*) }{ \partial \gamma^\top} - \frac{ \partial Q(\beta, \gamma^*) }{\partial \beta^\top} \left[ \frac{\partial F(\beta, \gamma^*)}{\partial \beta^\top} \right]^{-1}
\frac{\partial F(\beta, \gamma^*)}{\partial \gamma^\top }.
\]
To simplify notations, let
\[
A=\frac{ \partial Q(\beta, \gamma^*)}{ \partial \gamma^\top},~~B=\frac{\partial Q(\beta, \gamma^* ) }{ \partial \beta^\top },~~
V= \frac{ \partial F(\beta, \gamma^*) }{ \partial \beta^\top },~~D=\frac{ \partial F(\beta, \gamma^*) }{\partial \gamma^\top }.
\]
When emphasizing the arguments $\beta$ and $\gamma$, we write $A(\beta, \gamma^*)$ instead of $A$ and so on.
When $ \beta \in B( \beta^*, \epsilon_{n1})$, $V \in \mathcal{L}(b_{n0}, b_{n1})$.
Let $W=V^{-1}-S$, where $S=\mathrm{diag}(1/v_{11}, \ldots, 1/v_{nn})$.
Then we have
\[
H = A - B V^{-1} D = A - BSD - BWD.
\]
Recall that $z_*=\max_{i,j} \| z_{ij} \|_\infty$ and $\max_{i,j} | \mu_{ij}^{\prime}(\beta, \gamma^*) | \le b_{n1}$.
To simplify notations, we suppress the subscript ``$\max$" in the matrix maximum norm $\| \cdot \|_{\max}$ in this section.
It yields
\[
\|B\| \le n z_* b_{n1},~~~ \| D \| \le n  z_* b_{n1},
\]
such that
\begin{equation}\label{equation-v-remainder}
\|BWD\| = \max_{i,j} | B_{ik}W_{kl}D_{lj} | = O( \frac{b_{n1}^2}{n^2b_{n0}^3} ) \times n^2 z_*^2 b_{n1}^2 = O\left( \frac{b_{n1}^4 z_*^2 }{b_{n0}^3} \right).
\end{equation}

Now, we evaluate  $A(\beta, \gamma^*)-A(\beta^*, \gamma^*)$. By the mean value theorem, we have
\begin{eqnarray}
\nonumber
%%\| \frac{ \partial Q(\beta, \gamma^*)}{\partial \gamma} - \frac{\partial Q(\beta^*, \gamma^*)}{\partial \gamma } \|
\| A(\beta, \gamma^*) -  A(\beta^*, \gamma^*) \|
& = & \| \sum_{j\le i} z_{ij}z_{ij}^\top \left[ \frac{\partial \mu_{ij}(\beta, \gamma^*)}{\partial \pi_{ij}} - \frac{\partial \mu_{ij}(\beta^*, \gamma^*)}{\partial \pi_{ij}} \right] \| \\
\nonumber
& \le & \max_{i,j} \| z_{ij}z_{ij}^\top \| n^2 b_{n2} \| \beta - \beta^* \|_\infty \\
\label{verify-aa}
& \le & n^2 z_*^2  b_{n2} \| \beta - \beta^* \|_\infty.
\end{eqnarray}
Next, we evaluate $[BSD](\beta, \gamma^*)-[BSD](\beta^*, \gamma^*)$. Note that
\[
[BSD]_{ij}= \sum_{k, l} B_{ik} S_{kl} D_{lj} = \sum_{k=1}^n \frac{ B_{ik} D_{kj}}{ v_{kk} },
\]
\[
B_{kl}(\beta, \gamma^*)=\sum_{j\neq l} z_{ljk} \frac{ \partial \mu_{lj}(\beta, \gamma^*) }{\partial \pi_{lj} } ,
\]
\[
\frac{\partial B_{kl}(\beta, \gamma^*)}{\partial \beta} = \sum_{j\neq l} z_{ljk} \frac{ \partial^2 \mu_{lj}(\beta, \gamma^*) }{\partial \pi_{lj}^2 } T_{lj},
\]
\[
D_{kl}(\beta, \gamma^*)= \sum_{j\neq k} z_{kjl}\frac{ \partial \mu_{kj}(\beta, \gamma^*) }{ \partial \pi_{kj}},
\]
\[
\frac{ \partial D_{kl}(\beta, \gamma^*) }{ \partial \beta}= \sum_{j\neq k} z_{kjl} \frac{\partial^2 \mu_{kj}(\beta, \gamma^*)}{\partial \pi_{kj}^2 } T_{kj}.
\]
Since $|\mu^\prime(\pi_{ij})| \le b_{n1}$ and $|\mu^{\prime\prime}(\pi_{ij})| \le b_{n2}$,
we have
\begin{equation}
\label{equation-Bkl}
| B_{kl} (\beta, \gamma^*) | \le n\kappa_n b_{n1},~~ | D_{kl} (\beta, \gamma^*) | \le n z_* b_{n1},
\end{equation}
and for a vector $v$,
\begin{equation}
\label{equation-Bklv}
\begin{array}{rcl}
\| \frac{\partial B_{kl}(\beta, \gamma^*)}{\partial \beta^\top } v \| & \le & z_* b_{n2} [(n-1)|v_l| + \sum_{j\neq l} |v_j| ],
\\
\| \frac{ \partial D_{kl}(\beta, \gamma^*) }{ \partial \beta^\top} \| & \le & z_* b_{n2} [(n-1)|v_l| + \sum_{j\neq l} |v_j| ].
\end{array}
\end{equation}
By \eqref{equation-Bkl} and \eqref{equation-Bklv}, we have
\begin{equation}
\label{equation-verify-e}
|\frac{\partial [B_{ik}(\tilde{\beta}, \gamma^*)D_{kj}(\tilde{\beta}, \gamma^*)]}{\partial \beta^\top}v|
\le 2n z_*^2  b_{n1}b_{n2}[(n-1)|v_l| + \sum_{j\neq l} |v_j| ].
\end{equation}
It is easy to see that
\begin{equation}
\label{equation-verify-f}
|v_{kk}| \le \sum_{i\neq k} |\frac{ \partial \mu_{ik} }{ \partial \pi_{ik} }| \le (n-1)b_{n1},~~
\| \frac{ \partial v_{kk}}{\partial \beta^\top } \|_1 = \sum_{i\neq k} \|\frac{ \partial^2 \mu_{ik} }{ \partial \pi_{ik}^2 } T_{ik}\|_1
\le 2(n-1)b_{n2}
\end{equation}
By the mean value theorem, we have
\[
\frac{ B_{ik}(\beta, \gamma^*) D_{kj}(\beta, \gamma^*) }{ v_{kk}(\beta, \gamma^*) } - \frac{ B_{ik}(\beta^*, \gamma^*) D_{kj}(\beta^*, \gamma^*) }{ v_{kk}(\beta^*, \gamma^*) }
=f^\top(\tilde{\beta}, \gamma^*) (\beta - \beta^*),
\]
where $\tilde{\beta}$ lies between $\beta$ and $\beta^*$, and
\[
f(\beta, \gamma^*)= \frac{1}{v_{kk}^2(\tilde{\beta}, \gamma^*)}\left[\frac{\partial [B_{ik}(\tilde{\beta}, \gamma^*)D_{kj}(\tilde{\beta}, \gamma^*)]}{\partial \beta^\top}v_{kk}(\tilde{\beta}, \gamma^*)
- \frac{v_{kk}(\tilde{\beta}, \gamma^*)}{\partial \beta^\top}B_{ik}(\tilde{\beta}, \gamma^*)D_{kj}(\tilde{\beta}, \gamma^*)\right].
\]
By \eqref{equation-verify-e} and \eqref{equation-verify-f}, we have
\begin{eqnarray*}
&& |f^\top(\tilde{\beta}, \gamma^*) (\beta - \beta^*)|  \\
& \le & O\left(\frac{1}{(n-1)^2 b_{n0}^2 } \left\{[n^2\kappa_n^2  b_{n1}b_{n2} \| \beta - \beta^* \|_\infty  \times (n-1)b_{n1}
+ [n\kappa_n b_{n1}]^2 nb_{n2} \| \beta - \beta^* \|_\infty
\right\}\right) \\
& = & O( nb_{n1}^2 b_{n2} z_*^2 \|\beta - \beta^* \|_\infty b_{n0}^{-2} ) = O( n z_*^2 b_{n2} \|\beta - \beta^* \|_\infty \frac{b_{n1}^2}{b_{n0}^2} ).
\end{eqnarray*}
Consequently,
\begin{eqnarray}
\nonumber
&&|[BSD](\beta, \gamma^*)-[BSD](\beta^*, \gamma^*)| \\
\nonumber
& = & \sum_{k=1}^n |\left(\frac{ B_{ik}(\beta, \gamma^*) D_{kj}(\beta, \gamma^*)}{ v_{kk}(\beta, \gamma^*) }
- \frac{ B_{ik}(\beta, \gamma) D_{kj}(\beta, \gamma)}{ v_{kk}(\beta, \gamma) } \right)| \\
\label{inequality-verify-bb}
& \le & O( n^2 b_{n2}z_*^2 \|\beta - \beta^* \|_\infty b_{n1}^2b_{n0}^{-2} ).
\end{eqnarray}
By inequalities \eqref{equation-v-remainder}, \eqref{verify-aa} and \eqref{inequality-verify-bb}, if
\[
\frac{ b_{n1}^2}{b_{n0}^2} b_{n2}z_*^2 \|\beta - \beta^* \|_\infty = o(1),
\]
then
\[
\frac{1}{n^2} H(\beta, \gamma^*)_{ij} =\frac{1}{n^2} H(\beta^*, \gamma^*)_{ij} + o(1).
\]

\section{Simplifying expression of $B_*$ in (19)}
\label{section-simpli}
In the case of $V=\partial F(\beta^*, \gamma^*)/\partial \beta=\mathrm{Var}(d)$, $B_*$ can be simplified as follows.
Let $W=V^{-1} - S$. A direct calculation gives that
\[
  \sum_{k=1}^n \left[\frac{ \partial^2 Q(\beta^*, \gamma^*) }{ \partial \beta_k \partial \beta^\top}
 S  e_k \right] =
  \sum_{k=1}^n \frac{ \sum_{j\neq k} z_{kj} \mu_{kj}^{\prime\prime} (\pi_{ij}^*) }{ \sum_{j\neq k}  \mu_{kj}^{\prime} (\pi_{ij}^*) }.
\]
By Lemma 1, we have
\[
\sum_{k=1}^n \left[\frac{ \partial^2 Q_\ell(\beta^*, \gamma^*) }{ \partial \beta_k \partial \beta^\top}
 W  e_k \right] =  \sum_{j\neq k}z_{kj\ell} \mu^{\prime\prime}(\pi_{kj}^*) ( w_{kj} + w_{kn} )
 = O \left( \frac{ b_{n1}^2 b_{n2}}{ b_{n0}^3 n
 } \right).
\]
So, if $b_{n1}^2 b_{n2} b_{n0}^{-3} = o(n)$, then
\begin{equation}
\label{defintion-B-2}
B_* = \frac{1}{\sqrt{N}} \sum_{k=1}^n \left[\frac{ \partial^2 Q(\beta^*, \gamma^*) }{ \partial \beta_k \partial \beta^\top}
 V^{-1} e_k \right] = -\frac{1}{\sqrt{N}} \sum_{k=1}^n \frac{ \sum_{j\neq k} z_{kj} \mu_{kj}^{\prime\prime} (\pi_{ij}^*) }{ \sum_{j\neq k}  \mu_{kj}^{\prime} (\pi_{ij}^*) },
\end{equation}

\iffalse

\bibliography{reference3}

\bibliographystyle{apalike}
\addtolength{\itemsep}{-2 em}
\fi

\end{document}